\documentclass{amsart}
\usepackage{amssymb}
\usepackage{graphics,color}
\usepackage[latin1]{inputenc}
\usepackage[pagebackref=false]{hyperref}

\newtheorem{theorem}{Theorem}[section]
\newtheorem{thm}[theorem]{Theorem}
\newtheorem{lem}[theorem]{Lemma}
\newtheorem{prop}[theorem]{Proposition}
\newtheorem{cor}[theorem]{Corollary}
\theoremstyle{definition}
\newtheorem{defn}[theorem]{Definition}
\theoremstyle{remark}
\newtheorem{rem}[theorem]{Remark}
\newtheorem{ex}[theorem]{Example}
\newtheorem{que}[theorem]{Question}

\newcommand{\fm}{\mathfrak m}
\newcommand{\la}{\lambda}

\newcommand{\Ap}{\textrm{Ap}}
\newcommand{\Hi}{\textrm{H}}
\newcommand{\w}{\omega}
\newcommand{\ord}{\textrm{ord}}
\newcommand{\W}{\Omega}
\newcommand{\tord}{\textrm{tord}}

\begin{document}

\bibliographystyle{amsplain}

\title{On the Apéry sets of monomial curves }

\author[T. Cortadellas Ben\'{\i}tez] {Teresa Cortadellas Ben\'{\i}tez}
\address{ Teresa Cortadellas Ben\'{\i}tez\\Departament d' \'{A}lgebra i Geometria, Universitat de Barcelona, Gran Via 585,
08007 Barcelona, Spain.}

\email{terecortadellas@ub.edu}

\author[R. Jafari]{Raheleh Jafari}
\address{Raheleh Jafari\\School of Mathematics, Institute for
Research in Fundamental Sciences (IPM) P. O. Box: 19395-5746,
Tehran, Iran.}

\email{rjafari@ipm.ir}

\author[S. Zarzuela Armengou]{Santiago Zarzuela Armengou}
\address{ Santiago Zarzuela Armengou\\Departament d' \'{A}lgebra i Geometria, Universitat de Barcelona, Gran Via 585,
08007 Barcelona, Spain.}

\email{szarzuela@ub.edu}

\subjclass[2000]{13A30, 13H10, 13P10}

\keywords{numerical semigroup ring, Apéry set, monomial curve, tangent cone, Hilbert function, Gorenstein, Cohen-Macaulay, Buchsbaum.}

\thanks{Teresa Cortadellas Benítez and Santiago Zarzuela Armengou were supported by MTM2010-20279-C02-01.}

\thanks{Raheleh Jafari was supported in part by a grant from IPM (No. 900130068), by MTM2010-20279-C02-01, and by the IMUB of the University of Barcelona.}

\maketitle


\begin{abstract}
In this paper, we use the Apéry table of the numerical semigroup
associated to an affine monomial curve in order to characterize
arithmetic properties and invariants of its tangent cone. In
particular, we precise the shape of the Apéry table of  a numerical
semigroup of embedding dimension 3, when the  tangent cone of its
monomial curve is Buchsbaum or $2$--Buchsbaum, and give new proofs
for two conjectures raised by V. Sapko (Commun. Algebra
{29}:4759--4773, 2001) and Y. H. Shen (Commun. Algebra
{39}:1922--1940, 2001). We also provide a new simple proof in the
case of monomial curves for Sally's conjecture (Numbers of
Generators of Ideals in Local Rings, 1978) that the Hilbert function
of a one-dimensional Cohen-Macaulay ring with embedding dimension
three is non--decreasing. Finally, we obtain that monomial curves of
embedding dimension 4 whose tangent cones are Buchsbaum, and also
monomial curves of any embedding dimensions whose numerical
semigroups are balanced, have non--decreasing Hilbert functions.
Numerous examples are provided to illustrate the results, most of
them computed by using the NumericalSgps package of GAP (Delgado et
al., NumericalSgps-a GAP package, 2006).

\end{abstract}

\section{Introduction}

A numerical semigroup $S$ is a subset of the non-negative integers
$\mathbb N$, closed under addition, that contains $0$ and has finite
complement in $\mathbb N$. The condition $\# \,  \mathbb N \setminus S < \infty$ is equivalent to impose that $gcd(A)=1$ for every system of generators $A$ of $S$. Every numerical semigroup $S$ is finitely generated and has a unique minimal system of generators $n_1,\dots ,n_b$; that is
\[S=\langle n_1,\dots ,n_b \rangle :=\{ r_1n_1+\cdots+r_bn_b;
r_i\in \mathbb N \}\] with $n_1<\cdots <n_b$ and $n_i\notin \langle n_1,\dots , \widehat{n_i} , \dots , n_b \rangle$. Then we will say that $S$ is minimally generated by $\{n_1,\dots ,n_b\}$. The values $n_1$ and $b$ are known respectively as the multiplicity and the embedding dimension of $S$.

Let $S$ be a numerical semigroup minimally generated by $\{n_1,
\dots ,n_b\}$  and $k$ be a field. The ring
 \[k[[S]]:=k[[t^s;s\in S]]=k[[t^{n_1},\dots ,t^{n_b}]]\subset
 k[[t]],\] is called the numerical semigroup ring
 associated to $S$ which is a one-dimensional local domain with
 maximal ideal $\frak m=(t^{n_1},\dots ,t^{n_b})$,
 quotient field $k((t))$ and its  integral closure is the discrete valuation ring $k[[t]]$. So,
 $k[[S]]$ is an analytically irreducible and residually rational ring (i.e. the completion of the ring is an integral domain and the residue fields of the ring and of its integral closure coincide). This ring is indeed the coordinate ring of a monomial curve in the affine space $\Bbb{A}^{b}_k$ defined
 parametrically by $X_0=t^{n_1},\ldots, X_{b-1}=t^{n_b}$. We will denote by $v$ the t-adic valuation in $k[[t]]$. We
 will also consider the tangent cone associated to $k[[S]]$; that is
 the graded ring
 \[ G(S):=\bigoplus_{n\geq 0} \fm^n/\fm^{n+1}. \]

The main goal of this paper is to study arithmetic properties and invariants of the tangent cones of numerical semigroup rings. For that, we will use some Apéry sets associated to $S$, specifically, the Apéry sets with respect to $n_1$ of the ideals $nM$, where $M = S \setminus \{0\}$ is the maximal ideal of S and $n$ is a non-negative integer. In fact, only a finite number of such sets are needed, those corresponding to $nM$ with $0\leq n \leq r$, where $r$ is the reduction number of $S$. By using a result of V. Barucci and R. Fröberg \cite{BF1} about the existence of Apéry basis, we may organize these Apéry sets in a table that we call the Apéry table of $M$.

It has been shown by two of the authors of the present paper in \cite{CZ1} that this Apéry table provides precise information about the structure of $G(S)$ as a graded module over the fiber cone of the ideal generated by the minimal reduction $x = t^{n_1}$, which is a polynomial ring in one variable over $k$. This structure is in principle weaker than the structure of $G(S)$ as a ring itself, but as it was observed in the more general context of the study of the fiber cone of ideals with analytic spread one \cite{CZ3}, it provides enough information to determine several invariants and properties of the tangent cone, such as the regularity or the Cohen-Macaulay property. Moreover, in \cite{CZ2} the family of invariants given by this structure was explicitly related to other families of invariants of one-dimensional local rings like the microinvariants defined by J. Elias \cite{E2}, or the Apéry invariants defined by Barucci-Fröberg in \cite{BF}, see also \cite{BF1}.
One of the advantages of the Apéry table consists in its easy computation, avoiding for instance the study of explicit presentations of the tangent cone.

In this paper we go further in the applications of the Apéry table to determine new properties of tangent cones of numerical semigroup rings. They are mainly related to the $k$-Buchsbaum property and the behavior of the Hilbert function, and on the way, we also get simpler proofs of several known results concerning these properties. Of course, the use of Apéry sets to study the tangent cones of numerical semigroup rings is not new and it is explicit or implicitly contained in several papers of the extensive literature on the subject. Maybe one of the most well known results was first obtained by A. García \cite[pag. 403]{G} concerning the Cohen-Macaulayness of the tangent cone.

Now we describe the contents of this work. In Section 2, we first review some notions of numerical semigroup rings
and Apéry sets, and recall how to recover from the Apéry set of $S$ with respect to the multiplicity most of the basic invariants of the ring. Then we introduce our main tool in this paper, the Apéry table, and show how to read from it many properties of the tangent cone, in particular the Cohen-Macaulay and Gorenstein properties, or the Hilbert function of $k[[S]]$ and its behavior. Then we apply this method to study systematically the case of embedding dimension $2$ (plane monomial curves) and easily recover most of the known results in this case. We finish by studying the recent case considered by P. A García Sánchez, I. Ojeda, and J. C. Rosales \cite{GOR} of numerical semigroups with a unique Betti element, proving among other things that their tangent cone is Gorenstein.

The property for $G(S)$ to be Cohen-Macaulay has been studied
extensively by many authors, whereas the concept of a
Buchsbaum ring as the most important generalization of
Cohen-Macaulay rings, has not yet been so well characterized (or more in general, the $k$-Buchsbaum property). Some results for the Buchsbaum case by using Apéry sets may be found, for instance, in the paper by M. D'Anna, V. Micale and A. Sammartano \cite{DMS2}. On the other hand, in \cite{S}, V. Sapko, gave sufficient conditions for $G(S)$ to be Buchsbaum in embedding dimension $3$ and made two conjectures for equivalent conditions. These conjectures have been solved by Y. H. Shen \cite{Sh} and also by D'Anna-Micale-Sammartano \cite{DMS} independently
and by different methods. In Section 3, we first analyze in detail the case of embedding dimension $3$ with multiplicity $4$. Then, for the general case of embedding dimension three, we use the information given by Apéry table to describe completely the structure of $G(S)$ when it is Buchsbaum, providing as a consequence a new and simpler  proof of Sapko's conjectures. We also do the same in the more difficult $2$-Buchsbaum case, recovering also a result by Shen \cite{Sh} in this connection. We finish by showing that the above results cannot be extended to higher $k$-Buchsbaum and making an estimation of how this could be done.

In Section 4 we study the Hilbert function of $k[[S]]$. The problem of the behavior of the Hilbert function of a monomial curve, or more in general, of Cohen-Macaulay local rings of dimension one, has been studied by many authors from different points of view, in particular when it is non-decreasing. J. Sally \cite[page 40]{Sa} stated that loc. cit. \emph{it seems reasonable that the Hilbert function of a one-dimensional Cohen-Macaulay local ring with small
enough (say, at most three?) embedding dimension, is non-decreasing} (what is commonly known as Sally's conjecture). When the tangent cone is Cohen-Macaulay it is well known that the Hilbert function is non-decreasing, and so the result holds when the embedding dimension is $2$ (although the first proof of this fact was stated by E. Matlis \cite{M}). In \cite{E1}, Elias proved that Sally's statement is true in the equicharacteristic case of embedding dimension $3$, see also J. Elias and J. Martínez-Borruel \cite{EM} for a recent extension of this result to the non-equicharacteristic case. For higher embedding dimension the result is not true  in general. For instance, for any $d\geq 5$, there are examples by F. Orecchia \cite{O} of
reduced one-dimensional local rings of embedding dimension $d$ with
decreasing Hilbert function, and for embedding dimension $d=4$ there are examples by S. K. Gupta and L. G. Roberts \cite{GuRo}. But all these rings are not numerical semigroup
rings. For this special case, there are examples of embedding dimension $10$ with decreasing Hilbert functions by J. Herzog and R. Waldi in \cite{HW}, and of embedding dimension $12$ by P. Eakin and A. Sathaye in \cite{ES}. As far as we know, there are no examples of numerical semigroup rings of embedding dimension $4\leq
d\leq 9$ with decreasing Hilbert function. On the other hand, among the variety of positive results for numerical semigroup rings, we mention the recent paper by F. Arsalan, P. Mete and M. \c{S}ahin \cite{AMS} where some new cases are provided in order to support a conjecture by M. Rossi that the Hilbert function of a one-dimensional Gorenstein local ring of embedding dimension $4$ is non-decreasing. Their technique is based on the gluing of numerical semigroups. Our main result in this section is that the Hilbert function of a numerical semigroup ring of embedding dimension $4$ whose tangent cone is Buchsbaum is non-decreasing. We also get a very easy proof of Sally's conjecture for monomial curves in embedding dimensions $3$, and close this section by extending a recent result of D. P. Patil and G. Tamone \cite{PT} on the non-decreasing of Hilbert functions of balanced numerical semigroups.

\medskip

We provide many explicit examples along the paper. Most of the computations have done by using the NumericalSgps package of GAP \cite{Num}.

\section{The Apéry table of a monomial curve }

Throughout $S$ is a numerical semigroup minimally generated by
$\{n_1,\ldots,n_b\}$ with $n_1<\cdots<n_b$ and $k$ is a field. For the general notations and results about numerical semigroups and numerical semigroup rings we shall use the books by P. A. García Sánchez and J. C. Rosales \cite{GR} and by V. Barucci, D. E. Dobbs, and M. Fontana \cite{BDF}. A relative ideal of $S$ is a nonempty set $H$ of integers such that
$H+S\subset H$ and $d+H\subseteq S$ for some $d\in S$. An ideal of
$S$ is then a relative ideal of $S$ contained in $S$.  We denote by
$M$ the maximal ideal of $S$, that is, $M=S\setminus \{0\}$. If $L$
and $H$ are relative ideals of $S$ then $L+H=\{l+h;l\in L, h\in
 H\}$ is also a relative ideal of $S$. If $z\in \mathbb Z$, then
 $z+S=\{z+s; s\in S\}$ is the principal relative ideal of $S$
 generated by $\{z\}$ and if $z_1,\dots ,z_n\in \mathbb Z$, the relative ideal generated by
$\{z_1,\dots ,z_n\} \subset \mathbb Z$ is $(z_1+S)\cup \cdots \cup (z_n+S)$
that we will denote by $(z_1,\dots ,z_n)+S$. $M$ is
then the ideal generated by a system of generators of $S$. If $I$ is a
 fractional ideal of $k[[S]]$ then $v(I)$ is a relative ideal of $S$
 and for $J\subset I$ fractional ideals of $k[[S]]$, then $\la (I/J)=\# v(I)\setminus
 v(J)$, where $\la$ denotes the length. Thus, $v(k[[S]])=S$, $v(\fm^n )=nM=M+\stackrel{n}{\cdots} +M$ for any positive integer $n$, and if $I=(t^{i_1},\dots , t^{i_k}) \subset k[[S]]$ then $v(t^{i_1},\dots ,
 t^{i_k})=(i_1,\dots ,i_k)+S$.
\medskip

For $s\in S$ we consider the order of $s$, that we will denote by
$\ord(s)$, as the integer $k$ such that $s\in kM\setminus (k+1)M$.
In particular $s$ may be represented as $s=r_1n_1+\cdots+r_bn_b$
($r_i\geq 0$) with $\sum_{i=1}^br_i=k$ maximum over all
representations of $s$ in $S$. We call this representation a maximal
expression of $s$.

\begin{rem}
\label{rem1} Note that if $s=r_1n_1+\cdots+r_bn_b$ is a maximal
expression of $s$ then any subrepresentation of $s$ of the form
$s'=r'_1n_1+\cdots+r'_bn_b$ with $r'_i \leq r_i$ for all $i=1, \dots
, b$ is also a maximal expression of $s'$. We shall use this fact
frequently throughout this paper.
\end{rem}

If $s\in S$ has order $k$ then $t^s\in \fm^k \setminus \fm^{k+1}$
and we will denote by $(t^s)^{\ast} \in \fm^k/\fm^{k+1}
\hookrightarrow G(S)$ its initial form. Thus, we have a map
\begin{equation}
\begin{array}{lll}
\label{order}
S & \longrightarrow & G(S)\\
s & \mapsto & (t^s)^{\ast}
\end{array}
\end{equation}

Note that for two elements $s, s' \in S$, $(t^s)^*(t^{s'})^* =
(t^{s+s'})^*$ if and only if the sum of two respective maximal
expressions of $s$ and $s'$ is a maximal expression of $s + s'$.

\medskip

The  element $x=t^{n_1}$ generates a minimal reduction of $\frak
 m$ (equivalently $x$ is a superficial element of degree one. See \cite{Sa1} for the definitions of minimal reduction and superficial element.) and the following equalities hold for the local ring $A=k[[S]]$:

  \[
\begin{array}{l}
 e=n_1=\la (A/xA)=\# S\setminus
  (n_1+S)
  ,\\
  b=\la (\fm/\fm^2) =\# M\setminus
  2M,\\
    r= \min \{r\in \mathbb N \mid
  \fm^{r+1}=x\fm^r\}=
  \min \{r\mid
  (r+1)M= n_1+ rM \},\\
  \rho=\la (A[\frac{\fm}{x}]/A) =\# \langle n_1,n_2-n_1,\ldots, n_b-n_1\rangle \setminus S ,\\
  \mathfrak{c}=\min \{ c\in \mathbb N \mid c+n \in S \textrm{ for all } n\in
 \mathbb N \},\\
\delta =\la (\overline A/A)=\# \mathbb N\setminus S,\\
\tau =\la ((x\fm :\fm)/xA)=\la (\fm^{-1}/A)=\# \{n\in \mathbb Z\mid n+M
\in S\}\setminus S,
\end{array}
\]
where $e$ is the multiplicity of $A$, $b$ the embedding dimension,
$r$ the reduction number, $A[\frac{\fm}{x}]$ is the blow up ring of
$\fm$ in $A$, $ \rho $ the $N$-reduction number, $\overline{A}$ is the integral closure of $A$, $\mathfrak{c}$ is the conductor of $S$ and such that $C=t^\mathfrak{c}k[[t]]$ is the conductor of $A$ in $\overline A$, $\delta$ is
the degree of singularity of $A$ and $\tau$ is the Cohen-Macaulay type
of $A$. The following relations are  known for one dimensional
Cohen-Macaulay rings
\[\begin{array}{l}
b=e-\la(\fm^2 /x\fm) \leq e\\
r\leq e-1 \\
\rho=e-1+\sum_{i=1}^{r-1}\la(\fm^{i+1}/x\fm^i) \geq e-1\\
r\leq 1 \Leftrightarrow b=e \Leftrightarrow \rho =e-1
\end{array}
\]
and, since $A$ is analytically irreducible and residually rational,
also
\[
\begin{array}{l}
2(\mathfrak{c}-\delta)\leq \mathfrak{c} \leq (\tau +1)(\mathfrak{c}-\delta)\\
\tau =1 \Leftrightarrow \mathfrak{c}=2\delta
\end{array}
\]
(and recall that $A$ is Gorenstein exactly when $\tau =1$).

\medskip

We will denote by
\[
\Hi(n)=\mu(\fm^n)=\la(\fm^n/\fm^{n+1})=\# nM\setminus (n+1)M
\]
 the
Hilbert function of $k[[S]]$. We also recall that
$\Hi(n)=e-\la(\fm^{n+1}/x\fm^n)$ and, in particular, $e=\mu(\fm^n)$
for $n\geq r$.

\medskip

Let $H$ be a relative ideal  of $S$ and $n\in S$. We will denote by
$\Ap_{n} (H)$ the Apéry set of
 $H$ with respect to $n$; that is the set of the smallest elements in $H$ in each
 congruence class module $n$, equivalently, the set of elements $s$ in $H$ such $s-n\notin H$. In particular, $H = \Ap_{n} (H) + rn$, $r\in \mathbb{N}$. Then, the greatest integer not in $S$ (known as the Frobenius number) is
\[ \max \Ap_n(S)-n\]
and so the conductor is $c=\max \Ap_n(S)-n +1$. It can also be
deduced easily that the degree of singularity of $S$ (also called
the genus) can be calculated as
\[ \delta= \frac{1}{n} \left( \sum_{\w \in \Ap_n(S)}\w \right)
-\frac{n-1}{2},\]
see \cite[Proposition 2.12]{GR}.

\begin{rem}
\label{rem2} Observe that if $s=r_1n_1+\cdots+r_bn_b$ is an element
of $\Ap_n(S)$ then any element of the form
$s'=r'_1n_1+\cdots+r'_bn_b$ with $r'_i \leq r_i$ for all $i=1, \dots
, b$ is also an element of $\Ap_n(S)$.
\end{rem}

For the particular value $n=e=n_1$ we will write $\Ap_n(H) =
\Ap(H)$. Note that in this case $\{ n_2, \dots , n_b \} \subset
\Ap(S)$. Also that for any element $\w \in \Ap(S)$ there is a
maximal expression of the form $\w = \sum _{i=2}^{b}r_in_i$.

\medskip

 Put $W=k[[x]]=k[[t^e]]\subset A$. The following fact proved by Barucci-Fr\"{o}berg \cite[Lemma 2.1]{BF} in the more general setting of one-dimensional equicharacteristic analytically irreducible and residually rational domains, will be crucial for our results (see also \cite[Lemma 2.1]{CZ1} for a proof in the case of numerical semigroup rings): let $I$ be a fractional ideal of $A$ and  $f_0,\dots ,f_{e-1} $ elements in $I$ such that $\{v(f_0), \dots , v(f_{e-1})\} =\Ap(v(I))$. Then $I$ is a free $W$-module with basis $f_0, \dots
,f_{e-1}$.

\medskip

Apéry sets allow to compute most of the invariants introduced above
as the following example shows.

\begin{ex}
\label{s=<4,11,29>}
 Let $S=\langle  4,11,29 \rangle $. For this
semigroup we have multiplicity $e=4$ and embedding dimension $b=3$.
Then:
\[\begin{array}{l}
\Ap(M)=\{4,11,22,29\}\\
\Ap(2M)=\{8,15,22,33\}\\
\Ap(3M)=\{12,19,26,33\}\\
\Ap(4M)=\{16,23,30,37\}
\end{array}
\]
We have $4M=4+3M$ and $r=3$. Moreover
\[
\begin{array}{l}
 M\setminus 2M=\{ 4,11,29 \}\\
2M\setminus 3M=\{ 8,15,22 \}\\
3M\setminus 4M=\{ 12,19,26,33 \}
 \end{array}
\]
and hence the Hilbert function is $\Hi(n)= \{1,3,3,4 \rightarrow
\}$.

The numerical semigroup associated to the ring obtained by blowing
up $\fm$ in $A$ is $S'=\langle 4,7,25 \rangle =\langle 4, 7 \rangle$
and
\[\begin{array}{l}
\Ap(S)=\{0,11,22,29\}\\
    \Ap(S')=\{0,7,14,21\};
\end{array}\]
therefore $S'\setminus S=\{ 7,14,18,21,25 \}$ and $\rho=5$.

The conductor is $\mathfrak{c}=\max \Ap(S)-e+1=29-4+1=26$.

On the other hand, $\mathbb N \setminus
S=\{1,5,9,13,17,21,25,2,6,10,14,18,3,7\}$ and so $\delta =14$.

Since
\[ \begin{array}{ll}
A & = W\oplus Wt^{11}\oplus Wt^{22}\oplus Wt^{29},\\
xA & = Wt^4 \oplus Wt^{15}\oplus Wt^{26}\oplus Wt^{33} \\
\fm &=  Wt^4\oplus Wt^{11}\oplus Wt^{22}\oplus Wt^{29}
\end{array} \] we have that $(0:\fm)_{A/xA}$
\[
=(Wt^{22}\oplus Wt^{29}\oplus Wt^4 \oplus Wt^{15}\oplus
Wt^{26}\oplus Wt^{33})/(Wt^4 \oplus Wt^{15}\oplus Wt^{26}\oplus
Wt^{33})\]
\[\cong W/t^4W\oplus W/t^4W
\]
 and so $\tau =2$.
\end{ex}

Apéry sets also allow to detect when $S$ is symmetric, and so $A$ is
Gorenstein. Recall that, following E. Kunz \cite{Ku}, a numerical semigroup is said to be symmetric if there is an integer $k\notin S$ such that $s\in S$ if and only if $k-s\notin S$. Then, it was shown in \cite[Theorem]{Ku} that a one-dimensional analytically irreducible Noetherian local ring is Gorenstein if and only if its value semigroup is symmetric. Now, if for a given $n\in S$ we write the corresponding Apéry set as $\Ap_{n} (S) = \{0 < a_1 < \cdots < a_{n-1}\}$, then $S$ is symmetric if and only if $a_i + a_{n-i-1} =
a_{n-1}$ for all $i = 0, \dots, n-1$. This is equivalent to the existence of a unique maximal element in $\Ap(S)$ with respect to the natural order in $S$: $x \leq y$ if and only if $y = x + x'$ for some $x' \in S$.
For instance, note that among the two numerical semigroups appearing in the above example, only $S'$ is symmetric.
\medskip

We are now going to introduce the main device we shall use in this
paper. See \cite{CZ3} for the same construction in the more general
setting of the fiber cone of ideals with analytic spread one. Since
$t^eA\subset \fm$ is a minimal reduction, the graded ring
\[ F(t^e):=\bigoplus_{n\geq 0} (t^e)^nA/(t^e)^n\fm \cong
\bigoplus_{n\geq 0} (t^e)^nW/(t^e)^ {n+1}W
\]
is a polynomial ring in one variable over $k$. Moreover, the
extension
\[
F(t^e)\hookrightarrow G(S)\] is finite. As a consequence, since
$F(t^e)$ is a graded principal ideal domain and $G(S)$ is a finite
$F(t^e)$-module, we have a decomposition of $G(S)$ as a direct sum
of a graded finite free $F(t^e)$-module and a finite number of
modules of the form $(F(t^e)/((t^e)^{\ast})^cF(t^e))(k)$, where $k$
is an integer.

As it was noted in \cite{CZ3}, although the structure of $G(S)$ as a
ring is richer than its structure as $F(t^e)$-module, it holds that
$H^0_{G(S)_+}(G(S)) = T(G(S))$, the torsion of $G(S)$ as
$F(t^e)$-module. As a consequence, $G(S)$ is Cohen-Macaulay ring if
and only if $G(S)$ is free as $F(t^e)$-module. On the other hand, an
element of $G(S)$ belongs to the torsion $T(G(S))$ if and only if it
is annihilated by a power of $(t^{e})^{\ast}$. Thus in particular,
$(t^{s})^{\ast}$ is an element of torsion if and only if there
exists an integer $c>0$ such that $\ord(s+ ce)> \ord(s) +c$. Then,
if we define the subset of $S$
\[T :=\{s\in S;\, \exists \,c >0 \textrm{ with }  \ord(s+ c n_1)> \ord(s) +c  \}
\] the map in  (\ref{order}) sends
\[  T \mapsto  T(G(S)).\]
We also define, for $s\in T$, the torsion order of $s$ as
\[\tord(s)=\min \{c>0; \ord(s+cn_1)>\ord(s)+c\}.\]

\begin{rem}
\label{rem3} Note that if an element of $S$ of the form $s + cn_1$ belongs to $T$, then any element of the form $s + c'n_1$ also belongs to $T$, for any $0
\leq c' \leq c$. Also that if $\ord(s + cn_1)> \ord(s) + c$ then
$\ord(s + c'n_1) > \ord(s) + c'$ for any $c'\geq c$.
\end{rem}

\begin{rem}
\label{rem4} It is clear that the elements of the form $kn_1$, with
$k\geq 1$ are never torsion. It is also very easy to see that the
elements of the form $kn_2$, with $k\geq 1$ are never torsion.
\end{rem}

\begin{rem}
\label{rem5} Let $s$ be an element in $T$ with $c=\tord(s)$ and let $s+c n_1=\sum_{i=1}^bs_in_i$ be a maximal
expression of $s+c n_1$. Then it holds that $s_1=0$: otherwise we may write
$s+(c- 1)n_1= (s_1 - 1)n_1 + s_2n_2+\cdots +s_bn_b$, and $\ord(s+ (c- 1)n_1
)\geq (s_1 -1) + s_2+\cdots +s_b> \ord(s) +(c - 1)$ which contradicts the
minimality of $c$.
\end{rem}

In order to study the structure of $G(S)$ more in detail we recall
the considerations made in \cite{CZ1}, coming from an idea in
\cite{BF}. If we put, for $n\geq 0$,
\[ \Ap (nM)=\{\w_{n,0}=ne,\dots,\w_{n,i}, \dots ,\w_{n,e-1}\},\]
with $\w_{n,i}$ congruent to $i$ module $e$, then
\[ \fm^n= Wt^{\w_{n,0}} \oplus \cdots \oplus Wt^{\w_{n,i}} \oplus \cdots \oplus Wt^{\w_{n,e-1}},\]
and fixed $i$, $1\leq i \leq e-1$ one has  $\w_{n+1
,i}=\w_{n,i}+\epsilon \cdot e$ where $\epsilon \in \{0,1\}$ and
$\w_{n+1 ,i}=\w_{n,i}+ e$ for $n\geq r$. That is,
$\{t^{\w_{n,0}},\dots ,t^{\w_{n,e-1}}\}_{n\geq 0}$ is a family of stacked bases for the
free $W$-modules $\fm^n \subset A$. The table
\medskip

\[
\begin{array}{|c|c|c|c|c|c|c|}\hline
 \Ap(S)& \w_{0,0} &\w_{0,1}&\cdots  &
 \w_{0,i} &\cdots & \w_{0,e-1} \\
\hline
\Ap(M)& \w_{1,0} &\w_{1,1}&\cdots & \w_{1,i} & \cdots & \w_{1,e-1}\\
\hline \vdots& \vdots & \vdots &\vdots & \vdots & \vdots & \vdots
\\ \hline
\Ap(nM)&\w_{n,0} &\w_{n,1}&\cdots & \w_{n,i} & \cdots &
\w_{n,e-1}\\
\hline
\vdots & \vdots & \vdots &\vdots & \vdots & \vdots & \vdots \\
\hline
\Ap(rM) & \w_{r,0} &\w_{r,1}& \cdots & \w_{r,i} & \cdots  & \w_{r,e-1}\\
\hline  \end{array} \]

\medskip
\noindent is defined as the Apéry table of $M$ \cite{CZ1}.
\medskip

Before showing how to read the structure of $G(S)$ as
$F(t^e)$-module in the Apéry table, we recall the following notation
introduced in \cite{CZ1}.

Let $E = \{w_0, \dots , w_m\}$ be a set of integers. We call it a
stair if $w_0 \leq \cdots \leq w_m$. Given a stair, we say that a
subset $L = \{ w_i, \dots, w_{i+k}\}$ with $k\geq 1$ is a landing of
length $k$ if $w_{i-1} < w_i = \cdots = w_{i+k} < w_{i+k+1}$ (where
$w_{-1} = -\infty$ and $w_{m+1} = \infty$). In this case, the index
$i$ is the beginning of the landing: $s(L)$ and the index $i+k$ is
the end of the landing: $e(L)$. A landing $L$ is said to be a true
landing if $s(L) \geq 1$. Given two landings $L$ and $L'$, we set $L
< L'$ if $s(L) < s(L')$. Let $l(E)+1$ be the number of landings and
assume that $L_0< \cdots <L_{l(E)}$ is the set of landings. Then, we
define following numbers:

$s_j(E)= s(L_j)$, $e_j(E) = e(L_j)$, for each $0\leq j \leq l(E)$;

$c_j(E) = s_j - e_{j-1}$, for each $1 \leq j \leq l(E)$.

$k_j(E) = e_j - s_{j}$, for each $1 \leq j \leq l(E)$.

With this notation, for any $1\leq i \leq e-1$, consider the stairs
$\Omega ^i=\{\w_{n,i}\}_{0\leq n\leq r}$, that is, the stairs
defined by the columns of the Apéry table of $M$,  and the following
integers:

\begin{itemize}
\item $l_i = l(\W^i)$ the number of true landings of the column
$\W^i$;

\item $d_i = e_{l_i}(\W^i)$ the end of the last landing;

\item $b_j^i=e_{j-1}(\W^i)$  and $c_j^i=c_j(\W^i)$, for $1\leq
j\leq l_i$.
\end{itemize}
Observe that
\[
\ord(\w_{0,i})=e_0(\W^i)=b_1^i
\]
 is the place where the first landing
ends and
\[d_i=b_1^i+(c^i_1+k^i_1)+\cdots +(c^i_{l_i}+k^i_{l_i})\]
(in fact, $b_1^i$ is the invariant $b_i$ defined in \cite{BF} in the more general
context of one-dimensional equicharacteristic analytically irreducible and
residually rational domains, see also \cite{BF1}, and \cite{CZ2} for
their interpretation in the case of one-dimensional Cohen-Macaulay
rings).

It was proven in \cite{CZ1} that the  torsion submodule of $G(S)$ is
minimally generated by
\[
\left\{
 (t^{\w_{0,i}})^{\ast}, (t^{\w_{0,i}+c_1^ie})^{\ast},\dots ,
(t^{\w_{0,i}+(c_1^i+\cdots +c_{l_i-1}^i)e})^{\ast} \right\}_{\{1\leq
i \leq e-1;\, l_i\geq 1\}}
\]
so  $\tord(\w_{0,i}+c_j^ie)=c^i_{j+1}$, and the free submodule of
$G(S)$ admits the basis
\[
\left\{ (t^{\w_{0,i}+(c_1^i+\cdots +c_{l_i}^i)e})^{\ast}
\right\}_{1\leq i \leq e-1}.
\]

In particular, we have

\[G(S)\cong \left( F(t^e)\oplus \bigoplus_{i=1}^{e-1}
F(t^e)(-d_i) \right)\oplus
 \left(\bigoplus_{i=1}^{e-1} \bigoplus_{j=1}^{l_i} \frac{F(t^e)}{((t^e)^{\ast})^{c_j^i}F(t^e)}
(-b_j^i)\right).\]

\begin{rem}
\label{rem4} Note that as a consequence of the above description,
any $s \in T$ must come out as one of the elements in the Apéry
table of $M$.
\end{rem}

\begin{ex}
\label{table s=<4,11,29>}
 Let $S=\langle 4, 11,29\rangle $ be the semigroup in
\ref{s=<4,11,29>} and set the Apéry table of $M$
\[
\begin{array}{|c|c|c|c|c|}\hline  \Ap(S)&0&29&22&11\\ \hline
\Ap(M)&4&29&22&11\\ \hline \Ap(2M)&8&33&22&15\\ \hline
 \Ap(3M)&12&33&26&19\\ \hline \end{array}
 \]
In this case
\[\begin{array}{rl}
G(S)=& F(t^4)  \oplus F(t^4)\cdot(t^{33})^{\ast}  \oplus
F(t^4)\cdot(t^{22})^{\ast}  \oplus
F(t^4)\cdot(t^{11})^{\ast} \oplus
F(t^4)\cdot(t^{29})^{\ast}\\
\cong & F(t^4)  \oplus F(t^4)(-3)  \oplus  F(t^4)(-2) \oplus
F(t^4)(-1)  \oplus F(t^4)/(t^4)^{\ast}F(-1)\end{array}
\]

and $\Hi(n)=\{1,3,3,4\rightarrow \}$.
\end{ex}

\begin{ex}
\label{s=<5,6,14>}
 The Apéry table for $S=\langle 5, 6,14\rangle $ is
\[
\begin{array}{|c|c|c|c|c|c|}\hline  \Ap(S)&0&6&12&18&14\\ \hline
\Ap(M)&5&6&12&18&14\\ \hline \Ap(2M)&10&11&12&18&19\\ \hline
 \Ap(3M)&15&16&17&18&24\\ \hline \Ap(4M)&20&21&22&23&24\\ \hline\end{array}
 \]
For this semigroup the tangent cone (we will put $F=F(t^5)$) is
\[\begin{array}{rl}
G(S)=& F  \oplus F\cdot(t^{6})^{\ast}  \oplus
F\cdot(t^{12})^{\ast}  \oplus F\cdot(t^{18})^{\ast} \oplus
F\cdot(t^{24})^{\ast}  \oplus F \cdot (t^{14})^{\ast}\\
\cong & F  \oplus F(-1)   \oplus  F(-2) \oplus F(-3) \oplus F(-4) \oplus
F/((t^5)^{\ast})^2F(-1)\end{array}
\]
and $\Hi(n)=\{1,3,4,4,5\rightarrow \}$.
\end{ex}

\begin{ex} The Apéry table for $S=\langle 7, 8,17\rangle $ is
\[
\begin{array}{|c|c|c|c|c|c|c|c|}\hline  \Ap(S)&0&8&16&17&25&33&34\\ \hline
\Ap(M)&7&8&16&17&25&33&34\\ \hline \Ap(2M)&14&15&16&24&25&33&34\\
\hline
\Ap(3M)&21&22&23&24&32&33&41\\
\hline
\Ap(4M)&28&29&30&31&32&40&41\\
\hline
\Ap(5M)&35&36&37&38&39&40&48\\
\hline
\Ap(6M)&42&43&44&45&46&47&48\\
\hline
\end{array}
 \]
 So, the tangent cone (we will put $F=F(t^7)$) is
 \[\begin{array}{rl}
G(S)=& F \oplus F\cdot(t^{8})^{\ast}  \oplus F\cdot(t^{16})^{\ast}
 \oplus F\cdot(t^{24})^{\ast}\oplus F\cdot(t^{32})^{\ast}\oplus
F\cdot(t^{40})^{\ast}\oplus F\cdot(t^{48})^{\ast}\\
\oplus & F\cdot(t^{17})^{\ast}\oplus F\cdot(t^{25})^{\ast}\oplus
F\cdot(t^{33})^{\ast}\oplus
F\cdot(t^{34})^{\ast}\oplus F\cdot(t^{41})^{\ast}  \\
\cong & F  \oplus F(-1)   \oplus  F(-2) \oplus
F(-3)\oplus F(-4)\oplus F(-5)\oplus F(-6)\\
 \oplus & F/(t^7)^{\ast}(-1)\oplus F/(t^7)^{\ast}(-2)
\oplus F/(t^7)^{\ast}(-3)\oplus F/(t^7)^{\ast}(-2)\oplus
F/(t^7)^{\ast}(-4)
\end{array}
\]
and the  Hilbert function is $\Hi(n)=\{1,3,5,5,6,6,7\rightarrow \}$.
\end{ex}

\medskip

For $k\leq r$, the differences $\Hi(k)-\Hi(k-1)$ can be read in the
table as
\[
\# \{y\mid \ord(y)=k \textrm{ and } y \textrm{ is the end of a
landing } \}
\]
\[
- \# \{x \mid \ord(x)=k-1 \textrm{ and } \ord(x+n_1)>k \}
\]

As noted before, the ring  $G(S)$ is a Cohen-Macaulay ring if and
only if $G(S)$ is free as $F(t^e)$-module or equivalently, $T=0$.
The translation of this condition in the Apéry table of $M$ is that
there are no true landings in their columns (that is $l_i=0$ and
$d_i=b_i$ for $1\leq i \leq e-1$).

\begin{rem}\label{tord} Observe that for any $\w \in \Ap(S)$ we have that $\tord(w)< r-1$. Thus the ring $G(S)$ is Cohen-Macaulay  if and only if for any element $\w \in \Ap(S)$ one has
\[ \ord (\w + ce) = \ord (\w) +c \]
for all $ 0< c < r-1$. See also García \cite[Theorem 7,
Remark 8]{G}. As a consequence, $G(S)$ is Cohen-Macaulay for any
numerical semigroup $S$ with reduction number at most two.
\end{rem}

If $G(S)$ is a Cohen-Macaulay ring, then we can rewrite the
structure of $G(S)$ as $F(t^e)$-module in the form
\[G(S)\cong F(t^e)\oplus \bigoplus_{k=1}^{r}
F(t^e)(-k)^{\gamma_k}
 \]
with $\gamma_k:=\# \{i;\, b_i=k\}$. Moreover, if we order the Apéry
set in the form
\[ \Ap(S)=\{ \w_0=0, \w^1_1,\dots ,\w^1_{\gamma_1}, \dots,
\w^r_1,\dots ,\w^r_{\gamma_r}\} \] with $\ord(\w^k_i)=k$ the Apéry
table has the form

\[
\begin{array}{|c|c|c|c|c|c|c|c|c|}\hline
 \Ap(S)& 0 &\w^1_1&\cdots  &
 \w^1_{\gamma_1} &\cdots & \w^r_1& \cdots & \w^r_{\gamma_r} \\
\hline
\Ap(M)& e &\w^1_1&\cdots  &\w^1_{\gamma_1} &\cdots & \w^r_1& \cdots & \w^r_{\gamma_r} \\
 \hline
\Ap(2M)& 2e &\w^1_1+e&\cdots  &\w^1_{\gamma_1}+e &\cdots & \w^r_1& \cdots & \w^r_{\gamma_r} \\
 \hline
  \vdots& \vdots & \vdots &\vdots & \vdots
& \vdots & \vdots &\vdots &\vdots
\\ \hline
\Ap(kM)& ke & \w^1_1+(k-1)e& \cdots& \w^1_{\gamma_1}+(k-1)e &\cdots & \w^r_1& \cdots & \w^r_{\gamma_r} \\
 \hline
\vdots & \vdots & \vdots &\vdots & \vdots & \vdots & \vdots &\vdots &\vdots\\
\hline \Ap(rM)& re & \w^1_1+(r-1)e&\cdots &\w^1_{\gamma_1}+(r-1)e &\cdots & \w^r_1& \cdots & \w^r_{\gamma_r} \\
\hline  \end{array} \] and it is clear that for $k\leq r$
\[\mu(\fm^k)=\mu(\fm^{k-1})+\gamma_k \] and so, the Hilbert
function is non-decreasing since
\[H_S(n)=\{1,1+\gamma_1=b, 1+\gamma_1+\gamma_2,\cdots ,
 1+\gamma_1+\cdots +\gamma_r=e \rightarrow \}.\]

\begin{rem}\label{HiCM}
The numbers $\gamma_k$ were studied in \cite{CZ2} in a more general
context. Concretely
\[ \gamma_k=\la \left( \fm^k/(\fm^{k+1}+\fm^k\cap xA)\right)\]
and if $G(S)$ is a Cohen-Macaulay ring then
\[\gamma_k=\mu(\fm^k)- \mu(\fm^{k-1})=\la \left( \fm^k/(\fm^{k+1}+ x\fm^{k-1})\right);\]
so, in this case $\gamma_k>0$ for $k\leq r$ and the Hilbert function
\[ 1<\mu(\fm)<\mu(\fm^2)< \cdots <\mu(\fm^r)=\mu(\fm^n) \]
is strictly increasing until $r$ and stabilizes for $n\geq r$.
\end{rem}

\begin{rem}
\label{red}
Observe that if $G(S)$ is Cohen-Macaulay then the reduction number of $A$ equals to the highest order among the elements in $\Ap(S)$.
\end{rem}

The Gorenstein property of the tangent cone can also be detected in
terms of the Apéry table. According to L. Bryant \cite[Corollary
3.20]{B}, $G(S)$ is Gorenstein if and only if $G(S)$ is
Cohen-Macaulay and $S$ is symmetric and $M$-pure. Let us write the
Apéry set of $S$ in the form $\Ap(S) =\{ 0 = \w_0 < \w_1 < \cdots <
\w_{e-1}\}$ and assume $S$ is symmetric. Then, by \cite[Proposition
3.7]{B} $S$ is $M$-pure if and only if $\ord(\w_i) +
\ord(\w_{e-i-1}) = \ord(\w_{e-1})$ for all $i = 0, \dots, e-1$. In
terms of the Apéry table this is equivalent to the condition $b_i +
b_{e-i-1} = b_{e-1}$ for all $i = 0, \dots, e-1$, which is a kind of symmetry for the ends of the first landings in the Apéry table.

\medskip

  Next we show how to apply the above techniques to systematically study the already known case of embedding dimension two, that is, plane monomial curves. Let $S=\langle n_1,n_2 \rangle $ be the numerical semigroup minimally generated by $\{n_1,n_2\}$. We have multiplicity $e=n_1$ and the embedding dimension is  $b=2$.
\medskip

The Apéry table of $M$ is completely determined. In fact, after
reordering the elements in $\Ap(S)$ in increasing form, we have
\[
\begin{array}{|c|c|c|c|c|c|c|}\hline
 S& 0 &n_2&\cdots  & kn_2 &  \cdots & (n_1-1)n_2 \\
\hline
M& n_1 &n_2&\cdots  & kn_2 &  \cdots & (n_1-1)n_2 \\
\hline
 2M& 2n_1 &n_2+n_1&\cdots  & kn_2 &  \cdots & (n_1-1)n_2 \\
\hline
  \vdots& \vdots & \vdots &\vdots & \vdots
& \vdots & \vdots \\
 \hline
 kM& kn_1 &n_2+(k-1)n_1&
 \cdots &kn_2 &  \cdots & (n_1-1)n_2 \\
 \hline
(k+1)M& (k+1)n_1 &n_2+kn_1&
 \cdots &kn_2 +n_1&  \cdots & (n_1-1)n_2 \\
 \hline
\vdots & \vdots & \vdots &\vdots & \vdots & \vdots & \vdots  \\
\hline (n_1-1)M& (n_1-1)n_1 &n_2+(n_1-2)n_1&
 \cdots &kn_2 +(n_1-k-1)n_1&  \cdots & (n_1-1)n_2 \\ \hline   \end{array} \]
\medskip

Observe that $n_1M=n_1+(n_1-1)M$ and so the reduction number is
$r=n_1-1$ (the Apéry table is an square box).

Considering now $\Ap(S)$ and $\Ap_{n_1}(S')$ for $S'=\langle
n_1,n_2-n_1 \rangle$ the semigroup obtained by blowing up $\fm$ at
$S$; that is
\[
\begin{array}{|c|c|c|c|c|c|}\hline  S'&0&n_2-n_1&2(n_2-n_1)& \cdots &(n_1-1)(n_2-n_1)\\ \hline
S&0&n_2&2n_2& \cdots &(n_1-1)n_2\\ \hline
\end{array}
 \]
we obtain that the N-reduction number is in this case  $\rho
=1+2+\cdots +(n_1-1)=\frac{n_1(n_1-1)}{2}$.

As $(n_1-1)n_2$ is the greatest element in $\Ap(S)$, the conductor
is in this case $\mathfrak{c}=(n_1-1)n_2-n_1+1=n_1n_2-n_1-n_2+1$.
\medskip

In order to calculate the degree of singularity of $k[[S]]$ we
consider the table of $\Ap(S)$ and $\Ap_{n_1}(\mathbb N)$ or
equivalently the formula $\delta= \frac{1}{n} \left( \sum_{\w \in
\Ap_n(S)}\w \right) -\frac{n-1}{2}=\frac{n_1n_2-n_1-n_2+1}{2}$.
Thus, since $\mathfrak{c}=2\delta$ the ring $k[[S]]$ is Gorenstein. In fact
$S$ is clearly symmetric by just checking the Apéry set of $S$.

The Apéry table of $M$ shows that the tangent cone $G(S)$ is a free
$F(t^{n_1})$-module and so a Cohen-Macaulay ring; we can read in the
Apéry table that
\[
\begin{array}{rl}
G(S)&=F(t^{n_1})\oplus F(t^{n_1})\cdot (t^{n_2})^{\ast}\oplus
F(t^{n_1})\cdot (t^{2n_2})^{\ast}\oplus \cdots \oplus
F(t^{n_1})\cdot (t^{(n_1-1)n_2})^{\ast} \\ &\cong F(t^{n_1})\oplus
F(t^{n_1})(-1)\oplus F(t^{n_1})(-2)\oplus \cdots \oplus
F(t^{n_1})\cdot (-(n_1-1))
\end{array}
\]

Also from the Apéry table it is clear that $S$ is $M$-pure. So we
restate the following fact.

\begin{cor}\label{2}
If the embedding dimension of $S$ is 2, then $G(S)$ is Gorenstein.
\end{cor}

We finish this section by considering numerical semigroups with a
unique Betti element. Recently, García Sánchez-Ojeda-Rosales \cite{GOR} have studied affine semigroups having a
unique Betti element. In particular, they have characterized this
property for numerical semigroups in the following way \cite[Example
12]{GOR}: let $S$ be a numerical semigroup minimally generated by
$\{n_1, \dots , n_b\}$. Then, $S$ has a unique Betti element if and
only if there exist $k_1 > \cdots > k_b$ pairwise relatively prime
integers greater than one such that $n_i=\prod_{j\neq i}k_j$.

\begin{prop}
Let $S$ be a numerical semigroup with a unique Betti element. Then:

\begin{itemize}

\item[(1)] $\Ap(S) = \{\lambda_2n_2 + \cdots + \lambda_bn_b \mid 0 \leq \lambda_i \leq k_i-1\}$.

\item[(2)] $S$ is $M$-pure symmetric.

\item[(3)] $G(S)$ is Gorenstein.

\item[(4)] The reduction number $r$ is equal to $\sum _{i=2}^b(k_i-1)$.

\end{itemize}

\end{prop}

\begin{proof}

(1) First we show that $sn_i \in \Ap(S)$ if and only if $0 \leq s \leq k_i - 1$. In fact, if $s\geq k_i$ then $sn_i = k_in_i + (s-k_i)n_i = k_1n_1 + (s-k_i)n_i$ and so $sn_i\notin Ap(S)$. Conversely, let $l = \min \{s \mid sn_i\notin \Ap(S)\}$. Then, $ln_i = \sum_{i=1}^b r_in_i$ with $r_1 \neq 0$ and $r_i=0$. Since $r_i=0$, then $k_i|l$ and so $k_i \leq l$.

Now, since any subrepresentation of a representation of an element in $Ap(S)$ also belongs to $\Ap(S)$, it suffices to proof that the element $(k_2-1)n_2+ \cdots + (k_b-1)n_b \in \Ap(S)$. Assume the contrary and let $(k_2-1)n_2+ \cdots + (k_b-1)n_b = \sum _{i=1}^b s_in_i$, with $s_1\neq 0$ and $\sum _{i=2}^b s_in_i \in \Ap(S)$. Then, by the previous considerations, $0\leq s_i \leq k_i-1$ for $2 \leq i \leq b$. Thus we may write $s_1n_1 = \sum _{i=2}^b (k_i-1 -s_i)n_i$. Let $2\leq j \leq b$ be such that $k_j - 1 - s_j \neq 0$. Then $k_j|(k_j-1-s_j)$ because $k_j|n_i$ for all $i\neq j$, but this is impossible.

(2) From (1) we have that there is exactly one maximal element in the Apéry set, so we get that $S$ is symmetric. Hence to prove that $S$ is $M$-pure it is enough to show that $\sum_{i=2}^b(k_i-1)n_i$ is a maximal expression. Assume the contrary and let $\sum_{i=2}^b(k_i-1)n_i = \sum_{i=2}^bs_in_i$ with $\sum_{i=2}^bs_i >
\sum_{i=2}^b(k_i-1)$. Then, there exists $2\leq j \leq b$ such that $s_j > k_j-1$. Thus $s_jn_j \notin \Ap(S)$ as we have seen in (1), which is a contradiction.

(3) First we show that $G(S)$ is Cohen-Macauly. Assume that $G(S)$ is not Cohen-Macaulay and let $w\in \Ap(S)\cap
T\setminus (0)$ be the smallest possible torsion element in
$\Ap(S)$. Note that then $w$ is also the smallest possible element
in $T\setminus (0)$. In fact, let $x \in T\setminus (0)$ and $w'\in
Ap(S)$ such that $x = w' + kn_1, k\geq 0$. Then, $w'$ is also
torsion and so $w \leq w' \leq x$.

Assume that $\tord(w) = c$ and let $w=\sum^b_{i=2}r_in_i$ a maximal
expression, and $w+cn_1=\sum^b_{i=2}s_in_i$ a maximal expression
too (by Remark \ref{rem5} we know that $s_1 = 0$). Note that $s_ir_i=0$ for all $i=2,\ldots,b$, since $w$ is a
minimal element of $T$ (see Lemma \ref{subtorsion}). Let
$i\in\{2,\ldots,b\}$. If $s_i\neq 0$, then $r_i=0$ and so $k_i\mid
w$. If $s_i=0$, then $k_i\mid \sum^b_{i=2}s_in_i=w+cn_1$.
Since $k_i\mid n_1$ we have that $k_i\mid w$. Hence
$\prod_{i=2}^bk_i\mid w$, that is $n_1\mid w$ which contradicts
$w\in\Ap(S)$.

Finally, we get that $G(S)$ is Gorenstein because it is Cohen-Macaulay and $S$ is $M$-pure symmetric by (2).

(4) By Remark \ref{red} we have that $r$ is equal to the highest order among the elements in the Apéry set, which in this case is $\sum _{i=2}^b(k_i-1)$.
\end{proof}

\begin{rem}
By completely different methods it is proven in \cite{GOR} that $S$ is in fact complete intersection. Also, following the notation in \cite{DMS1}, we have that $\Ap(S)$ is $\beta$-rectangular and then $G(S)$ is complete intersection by \cite[Corollary 2.10 and Theorem 3.6]{DMS1}.
\end{rem}

\section{$k$--Buschsbaum property in the three-generated case}


Recall that $G(S)$ is called $k$--Buchsbaum if  and only if
$G(S)_+^k \cdot H^0_{G(S)_+}(G(S))=0$. In our case
$H^0_{G(S)_+}(G(S))$ coincides with $T(G(S))$, the torsion
$F(t^e)$--submodule of $G(S)$. The $1$--Buchsbaum condition is the
Buchsbaum condition and $G(S)$ is  $0$--Buchsbaum precisely when it
is Cohen-Macaulay.

\medskip

Observe that because in this case $H^0_{G(S)_+}(G(S))$ is of finite
length, $G(S)$ will be always $k$-Buchsbaum for some $k\leq \lambda
(H^0_{G(S)_+}(G(S)))$. Thus in some sense, the study of the
$k$--Buchsbaum property in this case is a sort of classification of
the family of numerical semigroups, being the Cohen-Macaulay case
the best possible from this point of view.

\medskip

In this section we use the Apéry table to provide the precise
structure of $G(S)$ when $S$ is three-generated and $G(S)$ is
Buchsbaum or $2$-Buchsbaum. As a biproduct we will have positive
answers to conjectures raised by Sapko in \cite{S} for the
Buchsbaum case and by Shen in \cite{Sh} for the $2$-Buchsbaum
case. Both conjectures have been solved positively by using
different techniques by D'Anna-Micale-Sammartano in
\cite{DMS}, and Shen \cite{Sh} for the Buchsbaum case, and by Shen
\cite{Sh} for the $2$-Buchsbaum case.



\begin{rem}\label{tordB}
If $G(S)$ is $k$--Buchsbaum and $s\in T$ then the order of torsion
of $s$ is at most $k$ since otherwise $\ord(s+ie)=\ord(s)+i$ for
$1\leq i \leq k$ and $(t^{ke})^{\ast}\cdot (t^s)^{\ast}\neq 0$.
\end{rem}

In studying the $k$--Buchsbaum property the differences between the
structure of $G(S)$ as a graded ring or as a graded module over
$F(t^e)$ appear. For instance, we may have that $F(t^e)_+ \cdot
H^0_{G(S)_+}(G(S))=0$ (equivalently, the order of torsion of any
element in the torsion $F(t^e)$-submodule of $G(S)$ is one) but
$G(S)$ may be not Buchsbaum, see for instance \cite[Example
4.6]{CZ1}.

\medskip

The following lemma which holds for any numerical semigroup $S$ has
a  key role in the proof of our main results.

\medskip

\begin{lem}\label{general}
Let $s=\sum_{i=1}^b r_in_i$ be a maximal expression of $s$. If $s\in
T$, then $\sum_{i=3}^br_i\neq 0$.
\end{lem}
\begin{proof}
Let $c=\tord(s)$ and $s+c n_1=\sum_{i=1}^bs_in_i$ a maximal
expression of $s+c n_1$. Then
\[\sum_{i=1}^bs_i>\sum_{i=1}^br_i +c.\]
with $s_1= 0$ by Remark \ref{rem5}. Hence it holds that $\sum_{i=3}^br_i\neq 0$: otherwise
$s=r_1n_1+r_2n_2$ and then
\[ \begin{array}{rl} s+c n_1=r_1n_1+r_2n_2+c n_1 &=(r_1+c )n_1
+r_2n_2\\
& <(r_1+r_2+c )n_2 \\ & <(\sum_{i=2}^bs_i)n_2
\\ &<\sum_{i=2}^bs_in_i=s+c n_1, \end{array}
\]
a contradiction.
\end{proof}

In the three-generated case we can make more precise estimations:

\begin{lem}
\label{rs} Let $b=3$ and $S$ minimally generated by $n_1 < n_2 < n_3$. Let
$s=r_1n_1+r_2n_2+r_3n_3$ be a maximal expression of $s$. Assume that
$s\in T$ and let $c=\tord(s)$ with $s+c n_1=s_1n_1+s_2n_2+s_3n_3$ a
maximal expression of $s+c n_1$. Then
\begin{enumerate}
\item $s_1=0$. \item $r_3\neq 0$. \item $s_2 > r_2$.
\end{enumerate}
\end{lem}
\begin{proof}
Items $(1)$ and $(2)$ are proved in Lemma \ref{general}.

If $s_2 \leq r_2$ then $r_1n_1+(r_2-s_2)n_2+r_3n_3 + cn_1=s_3n_3$ with
$r_3<r_1+(r_2-s_2)+r_3+c<s_3$. Consider the element $x=r_1n_1+(r_2-s_2)n_2$.
Then $\ord(x)=r_1+(r_2-s_2)$ and $x+c n_1=(s_3-r_3)n_3$ with $s_3-r_3 >
r_1+(r_2-s_2)+c$. Hence $r_1n_1+(r_2-s_2)n_2\in T$, which contradicts condition
$(2)$ for $x$.
\end{proof}

From now on we will assume that $b=3$, and $S$ minimally generated by $n_1 <
n_2 <n_3$. It is then easy to see that for any element $\w$ in the
Apéry set there is a unique maximal expression of the form $kn_2 +
hn_3$, and that this maximal expression occurs with the maximum
possible value of $k$ for a representation of $\w$. So we may order
the elements in the Apéry set as follows:
\[ \Ap(S)=\{0, n_2,\dots , h n_2, n_3, \dots , n_3+h_1n_2, \dots,
k_Sn_3,\dots , k_Sn_3+ h_{k_S}n_2 \} \] with
$\ord(kn_3+j_kn_2)=k+j_k$ for all $k = 0, \dots, k_S$ and $j_k = 0,
\dots , h_k$ (where $h_0=h$).

\begin{rem}
\label{Ap} Let $kn_3+h'n_2 \in \Ap(S)$ with $\ord(kn_3+h'n_2)=k+h'$.
If $k>0$ then $(k-1)n_3+h'n_2 \in \Ap(S)$ with
$\ord((k-1)n_3+h'n_2)=k-1+h'$; and if $h'>0$ then $kn_3+(h'-1)n_2
\in \Ap(S)$ with $\ord(kn_3+(h'-1)n_2)=k+h'-1$. In particular,
\[0\leq h_{k_S} \leq \cdots \leq h_1 \leq h .\]
\end{rem}

\begin{rem}
By Lemma \ref{rs}, $\{n_2, \dots ,hn_2\} \cap T =\emptyset$.
\end{rem}

Next two lemmas show in particular that the behavior of the set of
torsion elements in $\Ap(S)$ is rigid:

\begin{lem}
\label{kh} With the notations introduced, assume that we have
maximal expressions $kn_3$ and $kn_3 + h'n_2$ for some $k, h' \geq
1$. Then
\[kn_3\in T \Leftrightarrow kn_3+h'n_2 \in T\]
and $\tord(kn_3)=\tord(kn_3+h'n_2)$.
\end{lem}
\begin{proof}
It is clear that if $\ord(kn_3+c n_1)>k+c$ then $\ord(kn_3+h'n_2+c
n_1)>k+h'+c$ and that $\tord(kn_3)\geq \tord(kn_3+h'n_2)$.

Reciprocally, let $c = \tord(kn_3+h'n_2)$. By Lemma \ref{rs} we may
write $kn_3+h'n_2+c n_1=s_2n_2+s_3n_3$ with $s_2 > h'$ and
$s_2+s_3>k+h'+c$. Then $kn_3 + cn_1 = (s_2 - h')n_2 + s_3n_3$ with $s_2 - h' + s_3 > k + c$,
hence $kn_3\in T$ and $\tord(kn_3)\leq \tord(kn_3+h'n_2)$.
\end{proof}

\begin{lem}
\label{k} With the notations introduced, assume that we have maximal
expressions for some $1\leq k < k'$. Then
\[ kn_3\in T \Rightarrow k'n_3 \in T\]
$\tord(k'n_3)\leq \tord(kn_3)$
\end{lem}
\begin{proof}
Let $c = \tord(kn_3)$. By Lemma \ref{rs} we have that
$kn_3+cn_1=s_2n_2+s_3n_3$ with $s_2+s_3>k+c$. Then
$k'n_3+cn_1=kn_3+(k'-k)n_3+cn_1=s_2n_2+(s_3+k'-k)n_3$, so
$\ord(k'n_3+cn_1)\geq s_2+s_3+k'-k >k'+c$, so $k'n_3 \in T$ and
$\tord(k'n_3)\leq c = \tord (kn_3)$.
\end{proof}


\begin{cor}
\label{T} With the notations introduced let $k \in \{1,\dots
,k_S\}$.
\[T\cap \{ kn_3, \dots, kn_3 +h_k n_2 \} \neq  \emptyset
\Leftrightarrow kn_3 \in T.\]
\end{cor}

\begin{ex}
Let $S=\langle 8,11,18 \rangle $. The Apéry set is
\[ \Ap(S)=\{
0,n_2=11,2n_2=22,3n_2=33,n_3=18,n_3+n_2=29,2n_3=36,2n_3+n_2=47\}
\]
and we have that $2n_3+n_1=44=4n_2$; so $2n_3\in T$ but $n_3\notin
T$. Thus, the converse of Lemma \ref{k} does not hold. The Apéry table of $M$ is
in this case
\[
\begin{array}{|c|c|c|c|c|c|c|c|c|}\hline  \Ap(S)&0&11&18&22&29&33&36&47 \\ \hline
\Ap(M)&8&11&18&22&29&33&36&47 \\ \hline
\Ap(2M)&16&19&26&22&29&33&36&47 \\ \hline
\Ap(3M)&24&27&34&30&37&33&44&47 \\ \hline
\Ap(4M)&32&35&42&38&45&41&44&55 \\ \hline
\Ap(5M)&40&43&50&46&53&49&52&55\\
\hline
\end{array}
 \]
\end{ex}

\medskip

As a consequence we have a very simple method to decide when the
tangent cone is Cohen-Macaulay in this case:

\begin{prop}
\label{CM} Let $b=3$ and $S$ minimally generated by $n_1 < n_2 < n_3$. With the notations introduced,
\[ G(S) \textrm{ is a Cohen-Macaulay ring } \Leftrightarrow k_Sn_3
\notin T. \]
\end{prop}
\begin{proof}
$G(S)$ is Cohen-Macaulay if and only if $T(G(S))=0$ if and only if
$T = \emptyset$ if and only if $\Ap(S)\cap T=\emptyset$ and by the above lemmas if and
only if $k_Sn_3 \notin T$.
\end{proof}

\begin{ex}
Let $S=\langle 8,11,18 \rangle $. The Apéry set is
\[ \Ap(S)=\{
0,n_2=11,2n_2=22,3n_2=33,n_3=18,n_3+n_2=29,2n_3=36,2n_3+n_2=47\}
\]
and we have that $2n_3+n_1=44=4n_2$; so $2n_3\in T$ but $n_3\notin
T$. Thus, the converse of Lemma \ref{k} does not hold. The Apéry table of $M$ is
in this case
\[
\begin{array}{|c|c|c|c|c|c|c|c|c|}\hline  \Ap(S)&0&11&18&22&29&33&36&47 \\ \hline
\Ap(M)&8&11&18&22&29&33&36&47 \\ \hline
\Ap(2M)&16&19&26&22&29&33&36&47 \\ \hline
\Ap(3M)&24&27&34&30&37&33&44&47 \\ \hline
\Ap(4M)&32&35&42&38&45&41&44&55 \\ \hline
\Ap(5M)&40&43&50&46&53&49&52&55\\
\hline
\end{array}
 \]
\end{ex}

\medskip

If $n_1=3$ then $3 = b = e$ and it is well known by \cite[Theorem 2]{Sa1} that $G(S)$ is Cohen-Macaulay (observe that in this case the reduction number is $r\leq 1$ and from the Apéry table is also clear that $G(S)$ is Cohen-Macaulay). Next we see the case $n_1 = 4$.

\begin{cor}
\label{4} Let $S=\langle 4,n_2,n_3\rangle $. Then
\[ G(S) \textrm{ is not  a Cohen-Macaulay ring }\Leftrightarrow
n_3+4=3n_2\Leftrightarrow r=3.\]
\end{cor}

\begin{proof}
First of all we have that $r\leq 3$ and so $\tord(w)\leq 1$ for any
$w\in Ap(S)$, by Remark \ref{tord} . It is clear that if $n_3+4=3n_2$ then $n_3\in T$ and
$G(S)$ is not Cohen-Macaulay. On the other hand, the Apéry set of
$S$ is $\Ap(S)=\{0,n_2,n_3,n_3+n_2\}$ or
$\Ap(S)=\{0,n_2,2n_2,n_3\}$. In both cases, if $G(S)$ is not
Cohen-Macaulay then $n_3\in T$ by Proposition \ref{CM} and, since $\tord(n_3) =1$, it
is easy to deduce from Lemma \ref{rs} that $n_3+4=3n_2$.

For the right equivalence we observe that if $n_3+4=3n_2$ then
$n_3+n_2\notin \Ap(S)$ (since $n_3+n_2$ is equivalent to $0$ module
$4$). So by Lemma \ref{rs} there are three possibilities for the
Apéry table of $S$:
\[
\begin{array}{|c|c|c|c|c|}\hline  \Ap(S)&0&n_2&n_3&n_3+n_2\\ \hline
\Ap(M)&4&n_2&n_3&n_3+n_2\\ \hline \Ap(2M)&8&n_2+4&n_3+4&n_3+n_2\\
\hline
 \end{array} \quad
\begin{array}{|c|c|c|c|c|}\hline  \Ap(S)&0&n_2&2n_2&n_3\\ \hline
\Ap(M)&4&n_2&2n_2&n_3\\ \hline \Ap(2M)&8&n_2+4&2n_2&n_3+4 \\
\hline
  \end{array}
  \]
  \[
  \begin{array}{|c|c|c|c|c|}\hline  \Ap(S)&0&n_2&2n_2&n_3\\ \hline
\Ap(M)&4&n_2&2n_2&n_3\\ \hline \Ap(2M)&8&n_2+4&2n_2&n_3+4\\ \hline
\Ap(3M)&12&n_2+8&2n_2+4&n_3+4 \\ \hline
  \end{array}
 \]
The table at the bottom corresponds to the only non Cohen-Macaulay
case (that is, when $n_3+4=3n_2$) and $r=3$ in this case. Otherwise
$r=2$.
\end{proof}

\begin{ex}
$\langle 4,5,6 \rangle$, $\langle 4,5,7 \rangle$ and $\langle 4,5,11
\rangle$ are semigroups each of them corresponding to each one of
the Apéry tables in the above corollary.
\end{ex}

\begin{rem}
In the situation of Corollary \ref{4}, we observe that if $G(S)$ is
not Cohen-Macaulay then the torsion $T(G(S))$ is minimally generated
by $(t^{n_3})^{\ast}$. Since $\tord(n_3)=1$ we have that $T(G(S))=\{ 0, (t^{n_3})^{\ast}  \}$, and so it has length $1$ and
$G(S)$ is then Buchsbaum because $G(S)_{+}\cdot T(G(S)) = 0$.
\end{rem}

Next we prove that in fact the above situation is the only
possibility for $G(S)$ to be Buchsbaum when $b=3$. As a consequence,
we get that if $b=3$ then $G(S)$ is Buchsbaum if and only if $\la
(T(G(S))) \leq 1$. This result was conjectured by Sapko \cite{S} and has
been proved by Shen in \cite{Sh} and also by
D'Anna-Micale-Sammartano in \cite{DMS} using different methods.

\begin{thm}
\label{buch} Assume that $b=3$. With the notations introduced,
\[ G(S) \textrm{ is  a Buchsbaum, not Cohen-Macaulay ring }
\Leftrightarrow T=\{k_Sn_3\} \]
\end{thm}

\begin{proof}
If $T=\{k_Sn_3\}$ then $T(G(S))$ is minimally generated by
$(t^{k_Sn_3})^{\ast}$, and because $\tord (k_Sn_3) = 1$ we get that
$T(G(S))$ has length $1$ and so, as we have noted at the beginning of this section, $G(S)$ is Buchsbaum.

Reciprocally, assume that $G(S)$ is Buchsbaum and not
Cohen-Macaulay. By Proposition \ref{CM} the element $k_Sn_3$ belongs
to $T$. We claim that $\Ap(S) \cap T= \{k_Sn_3 \}$. First observe
that $\tord(k_Sn_3) = 1$ by Remark \ref{tordB}. Hence, if $k_Sn_3+h'n_2 \in \Ap(S)\cap T$
with $\ord(k_Sn_3+h'n_2)=k_S+h'$ then $h'n_2$ has order $h'$ and
\[(t^{h'n_2})^{\ast} \cdot
(t^{k_Sn_3})^{\ast}=\overline{t^{k_Sn_3+h'n_2}}\in
\fm^{k_S+h'}/\fm^{k_S+h'+1}
\]
is not zero, which contradicts the Buchsbaum property of $G(S)$.
Now, by Lemmas \ref{kh} and \ref{k} we only have to prove that
$kn_3\notin T$ for any $k<k_S$. Assume that $kn_3\in T$ for some
$k<k_S$. Then
\[(t^{n_3})^{\ast} \cdot
(t^{kn_3})^{\ast}=\overline{t^{(k+1)n_3}}\in \fm^{k+1}/\fm^{k+2}
\]
is not zero because $(k+1)n_3 \in \Ap(S)$ with $\ord(k+1)n_3=k+1$,
which is again a contradiction because $G(S)$ is Buchsbaum.

Now, since $\Ap(S)\cap T=\{k_Sn_3\}$ we have that
\[ T\subset \{k_Sn_3+cn_1; c\geq 0\}.\]
On the other hand, $k_Sn_3+n_1=s_2n_2+s_3n_3$ with $\ord(s_2n_2 +
s_3n_3) = s_2+s_3 > k_S+1$.

We claim that $s_3=0$. Otherwise it is easy to see that
$(k_S-1)n_3\in \Ap(S) \cap T$ which contradicts the first claim.

Hence $k_Sn_3+n_1=s_2n_2$ which implies, by  Lemma \ref{general},
that $k_Sn_3+n_1\notin T$. As a consequence, $k_Sn_3 + cn_1\notin T
$ for all $c\geq 1$ and $T=\{k_Sn_3\}$ as we wanted to prove.
\end{proof}

\begin{cor}
Assume that $b=3$ and that $G(S)$ is Buchsbaum and not Cohen-Macaulay. With the notations
introduced, the Apéry set of $S$ has the form
\[ \Ap(S)=\{0, n_2,\dots , h n_2, \dots , kn_3+j_kn_2 , \dots ,
k_Sn_3 \} \] with $\ord(kn_3+j_kn_2)=k+j_k$ for all $k = 0, \dots,
k_{S}-1$ and $j_k = 0, \dots , h_k$ (where $h_0=h$) and
$\ord(k_Sn_3) = k_S$, and $G(S)$ is of the form
\[
G(S)\cong \left( F(t^{n_1}) \oplus \bigoplus_{k=0}^{k_{S}-1}
\bigoplus_{j_k=0}^{h_k} F(t^{n_1})(-k-j_k) \oplus F(t^{n_1})(-\alpha)\right)\oplus
 \frac{F(t^{n_1})}{((t^{n_1})^{\ast})F(t^{n_1})}
(-k_S),
\]
\noindent
where $\alpha = \ord (k_Sn_3+n_1)$.
\end{cor}

With similar ideas we can also describe the case $2$-Buchsbaum. As a
consequence we will have that if $b=3$, then $G(S)$ is $2$-Buchsbaum
(and not Buchsbaum) if and only if $\la(T(G(S)) = 2$. With different
techniques this result has also been proved by Shen \cite{Sh}. Note that in our case we may give an explicit description of $T(G(S))$ in terms of the set $T$ which can be easily detected in the Apéry table.

\begin{thm}
\label{2b}
Assume that $b=3$. With the notations introduced,
\[ G(S) \textrm{ is  $2$-Buchsbaum and not $1$-Buchsbaum}
\Leftrightarrow
\begin{cases} T=\{k_Sn_3,k_Sn_3+n_1 \} \textrm{ or } \\
T=\{k_Sn_3,k_Sn_3+n_2 \}, \textrm{ or } \\
T=\{k_Sn_3,(k_S-1)n_3 \}
 \end{cases}
\]
\end{thm}
\begin{proof}
$\Leftarrow$ is clear since by similar arguments as in Proposition
\ref{buch} we have that $\lambda(T(G(S))))=2$. Note that in the
structure of $G(S)$ there will be only one torsion direct summand of
order $2$ in the first case, and two torsion direct summands of
order $1$ in the other two cases.

We assume now that $G(S)$ is $2$--Buchsbaum and not Buchsbaum. Then
the torsion order of any $s\in S$ is at most $2$ by Remark \ref{tordB} and $\{k_Sn_3\}
\subsetneq T$ by Proposition \ref{CM} and Theorem \ref{buch}. We will distinguish three cases.

\underline{First case}: We assume that $\Ap(S)\cap T=\{k_Sn_3\}$.
We will prove that in this case the torsion order of $k_Sn_3$ is
equal to $2$.

If $\tord(k_Sn_3)=1$ then by Lemma \ref{rs},
$k_Sn_3+n_1=s_2n_2+s_3n_3$ with $\ord (s_2n_2+s_3n_3) =
s_2+s_3>k_S+1$. If $s_3\neq 0$ then $k_S>1$ because $\{n_1, n_2,
n_3\}$ is a minimal system of generators. But then $(k_S-1)n_3 \in
\Ap(S)\cap T$ (see Lemma \ref{subtorsion}) which contradicts the hypothesis. Thus
$k_sn_3+n_1=s_2n_2$, which does not belong to $T$ by Lemma
\ref{general}. So $k_sn_3+cn_1 \notin T$ for any $c\geq 1$ and
$T=\{k_Sn_3\}$. Hence the tangent cone $G(S)$ is Buchsbaum, a
contradiction.

Now, similarly to the above argument, since the torsion order of
$k_Sn_3$ is equal to $2$ and $(k_S-1)n_3$ is not a torsion element,
we may write $k_Sn_3+2n_1=s_2n_2\notin T$. Thus, $k_Sn_3+cn_1 \notin
T$ for any $c\geq 2$ and the only elements in the Apéry table of $M$
which are torsion are $0, k_Sn_3, k_Sn_3+n_1$. It is then clear that
$T=\{k_Sn_3,k_Sn_3+n_1 \}$.

\underline{Second case}: We assume that $\{k_Sn_3\}\subsetneq
\Ap(S)\cap T$ and $k_Sn_3+n_2 \in \Ap(S)$.

Note that $k_Sn_3+n_2$ is a maximal expression because $k_S$ is the
maximum possible value for such a representation.

First, we prove that $(k_S-1)n_3\notin T$. If $(k_S-1)n_3$ is a
torsion element, then because $G(S)$ is $2$--Buchsbaum we have that
$k_Sn_3+n_2= (k_S-1)n_3 + n_3 + n_2 = s_2n_1 + s_2n_2+s_3n_3$ with
$\ord(sn_1 + s_2n_2+s_3n_3) = s_1 + s_2 + s_3>k_S+1$, which is a
contradiction because $k_Sn_3+n_2$ is a maximal expression.

Secondly, by Lemma \ref{kh} we have that $k_Sn_3+n_2
\in T$, and that $\tord(k_Sn_3) = \tord(k_Sn_3+n_2)$. On the other
hand, $k_Sn_3+n_2+n_1$ is not a maximal expression because $G(S)$ is
$2$-Buchsbaum. Hence $\tord(k_Sn_3+n_2) = 1$ and  $\tord(k_Sn_3) =
1$ as well.

Finally, we prove that there are no more elements in the Apéry table
belonging to $T$. For that it suffices to see that $k_Sn_3 + n_1$
and $k_Sn_3 + n_2 + n_1$ are not in $T$. Because of Lemma \ref{rs},
a maximal expression of $k_Sn_3 + n_1$ must be of the form $k_Sn_3 +
n_1 = s_2n_2 + s_3n_3$ with $s_2 + s_3 > k_S + 1$. But then, if $s_3
\neq 0$ we would have that $(k_S-1)n_3\in T$, a contradiction. Hence
$s_3=0$ and by Lemma \ref{general}, $k_Sn_3 + n_1 \notin T$. A
similar argument shows that $k_Sn_3 + n_2 + n_1 \notin T$ (note that
$(k_S-1)n_3 + n_2\notin T$ by Lemma
\ref {kh} because $(k_S-1)n_3\notin T$). Thus we get $T=\{k_Sn_3, k_Sn_3+n_2\}$.

\underline{Third case}: We assume that $\{k_Sn_3\}\subsetneq
\Ap(S)\cap T$ and $k_Sn_3+n_2 \notin \Ap(S)$. Then, by Lemmas
\ref{kh} and \ref{k}, $(k_S-1)n_3\in \Ap(S)\cap T$. Note that
$k_S\geq 2$.

First we show that $\tord(k_Sn_3)=1$.  Assume the contrary. Then,
note that the element $n_3+n_1\in 2M\setminus 3M$ since otherwise,
by Lemma \ref{k}, $n_3\in T$ and $1=\tord(n_3)=\cdots
=\tord(k_Sn_3)$ which supposes a contradiction. But then we have
that
\[ 0\neq (t^{n_3+n_1})^{\ast}\cdot (t^{(k_S-1)n_3})^{\ast}=\overline{t^{k_Sn_3+n_1}}
\in \fm^{k_S+1}/\fm^{k_S+2} \] with $(t^{n_3+n_1})^{\ast}\in
G(S)_+^2$, which contradicts the assumption that $G(S)$ is
2--Buchsbaum.

Secondly, note that $kn_3\notin T$ for $k\notin \{k_S,k_S-1\}$,
otherwise $(k_S-2)n_3\in T$ and
\[ 0\neq (t^{2n_3})^{\ast} \cdot (t^{(k_S-2)n_3})^{\ast}
=\overline{t^{k_S}}\in \fm^{k_S}/\fm^{k_S+1}\,,
\]
again a contradiction to the fact that $G(S)$ is 2--Buchsbaum.

Now we prove that $(k_S-1)n_3 + n_2 \notin \Ap(S)$. Assume the
contrary: by Lemma \ref{kh} this will imply that $(k_S-1)n_3 +
n_2\in T$. Since $G(S)$ is $2$--Buchsbaum we have that
$(k_S-1)n_3+2n_2 = r_1n_1 + r_2n_2 + r_3n_3$, a maximal expression
with $r_1+r_2+r_3 > k_S+1$. We want to see first that $r_1\neq 0$.
If $r_1=0$ then $(k_S-1)n_3+2n_2 = r_2n_2 + r_3n_3$, a maximal
expression with $r_2+r_3 > k_S+1$. We have two possibilities: if
$k_S-1<r_3$ then $2n_2 = r_2n_2 + (r_3 - (k_S-1))n_3 > 2n_2$ because
$r_2 + r_3 -(k_S-1) > 2$, a contradiction. If $k_S-1 \geq r_3$ then
$((k_S-1)-r_3)n_3 = (r_2-2)n_2$ which implies that
$r_2-2>(k_S-1)-r_3$ because $n_3>n_2$. Consequently,
$((k_S-1)-r_3)n_3$ is not a maximal expression, which is a
contradiction to our set up in the order of the elements of
$\Ap(S)$. Thus we have that $r_1\neq 0$ as wanted to see. On the
other hand, $r_2=0$ because if not we would have that $(k_S-1)n_3 +
n_2$ is not a maximal expression and so this element cannot belong
to the Apéry set, which contradicts our hypothesis. Hence we have
$(k_S-1)n_3 + 2n_2 = r_1n_1 + r_3n_3$, with $r_1+r_3 > k_S+1$ and
$r_1\neq 0$. If $r_3 >k_S-1$ then $2n_2 = r_1n_1 + (r_3-(k_S-1))n_3$
with $r_1 + r_3-(k_S-1 )>2$, and so $2n_2$ is not a maximal
expression. But this is a contradiction: since $k_Sn_3\in T$ and has
order $1$ we have that by Lemma \ref{rs} that $k_Sn_3 + n_1 = s_2n_2
+ s_3n_3$, a maximal expression with $s_2 + s_3 > k_S +1$. We have
$s_3\leq 1$ because $kn_3\notin T$ for $k< k_S-1$, so $s_2 >
k_S\geq 2$ and this means that $2n_2$ is a maximal expression. Hence
we must have that $r_3\leq k_S-1$. Now we get the two following
equalities:
\begin{eqnarray}
\nonumber
((k_S-1) - r_3)n_3 + 2n_2 & = & r_1n_1 \\
\nonumber k_Sn_3 + n_2 & = & r'_1n_1
\end{eqnarray}
(the second one because $k_Sn_3+n_2\notin \Ap(S)$, $k_Sn_3\in
Ap(S)$, and $(k_S-1)n_3+n_2\in \Ap(S)$ by hypothesis). Subtracting
the first one to the second one we have that $(1+r_3)n_3 -n_2 =
(r'_1-r_1)n_1$ (and so $r'_1-r_1>0$), equivalently, $(1+r_3)n_3 =
(r'_1-r_1)n_1 + n_2$. But since $1+r_3\leq k_S$, the element
$(1+r_3)n_3\in \Ap(S)$ and so it must happen that $r'_1 -r_1 = 0$.
Hence $n_2 = (1+r_3)n_3$, a contradiction.

Our next step is to prove that $\tord((k_S-1)n_3)=1$. Assume the
contrary. Then it must be $2$ because $G(S)$ is $2$--Buchsbaum. By
Lemma \ref{rs} this implies that $(k_S-1)n_3 + 2n_1 = r_2n_2 +
r_3n_3$, a maximal expression with $r_2+r_3>k_S+1$. Then $r_3=0$
because $kn_3\notin T$ for $k<k_S-1$. Hence we have the equality
$(k_S-1)n_3+2n_1 = r_2n_2$, a maximal expression with $r_2>k_S+1$.
On the other hand, because the torsion order of $k_Sn_3$ is $1$ we
have again by Lemma \ref{rs} that $k_Sn_3+n_1 = r'_2n_2 + r'_3n_3$,
a maximal expression with $r'_2+r'_3>k_S+1$. Now, $r'_3=0$ because
we are assuming that $\tord((k_S-1)n_3=2$. Thus we have the two
equalities
\begin{eqnarray}
\nonumber
(k_S-1)n_3 + 2n_1 & = & r_2n_2 \\
\nonumber k_Sn_3 + n_1 & = & r'_2n_2
\end{eqnarray}
Subtracting the first one to the second one we get $n_3-n_1=
(r'_2-r_2)n_2$ (and so $r'_2-r_2>0$), equivalently $n_3 = n_1 +
(r'_2-r_2)n_2$, a contradiction.

We have proved that $T\cap \Ap(S) = \{(k_S-1)n_3, k_Sn_3\}$, both
elements with torsion order one. To finish, we must see that there
are no more torsion elements in $S$, and for that it will suffice to
see that there are no other  elements in the Apéry table which are
torsion. Since they must be congruent with $k_S$ or $k_S-1$ it is
enough to show that $k_Sn_3 + n_1, (k_S-1)n_3+n_1 \notin T$.

Observe first that by Lemma \ref{rs} we have an equality $(k_S-1)n_3
+ n_1 = s_2n_2 + s_3n_3$, a maximal expression with $s_2+s_3>k_S$.
Then, $s_3=0$ because $kn_3\notin T$ for $k<k_S-1$. Thus we have
$(k_S-1)n_3 + n_1 = s_2n_2$. But the elements of the form $kn_2$ are
never torsion for any $k\geq 1$, so $(k_S-1)n_3 + n_1 \notin T$.

Finally assume that $k_Sn_3+n_1\in T$. Again by Lemma \ref{rs} we
have an equality $k_Sn_3 + n_1 =  s_2n_2 + s_3n_3$, a maximal
expression with $s_2, s_3\neq 0$ and $s_2+s_3>k_S+1$. Observe that
$n_2 + n_3$ is a maximal expression. If $s_3 > k_S$ then $n_1 =
s_2n_2 + (s_3 - k_S)n_3$, which is impossible. So $k_S \geq s_3$.
Observe that by Lemma \ref{k}, $s_3n_3 \in T$ and hence $s_3 = k_S$
or $k_S-1$. On the other hand, $(k_S-s_3)n_3 + n_1 = s_2n_2$, with
$s_2 > k_S - s_3 + 1$, which means that $(k_S-s_3)n_3$ is torsion.
Hence $k_S - s_3 = k_S$ or $k_S -1$. Because $s_3 \neq 0$ we get
$k_S - s_3 = k_S -1$, and so $k_S = 2$ since $k_S\geq 2$. Thus $n_3
+n_2 \notin \Ap(S)$ and we must have $n_3 + n_2 = k_1n_1 + k_2n_2 +
k_3n_3$, with $k_1 \geq 1$. Observe then that $k_2 = k_3 = 0$, and
because $n_3 + n_2$ is a maximal expression, $k_1 \leq 2$. But this
is a contradiction because $2n_1 < n_2 + n_3$.
\end{proof}

\begin{ex}
The semigroups $S_1=\langle 5,6,14 \rangle$, $S_2=\langle 8,11,18 \rangle$, and $S_3=\langle 10,16,27 \rangle$ are
$2$--Buchsbaum and not Buchsbaum. Following the cases in the above
proof, we have that $S_1=\langle 5,6,14 \rangle$ belongs to the first case,
$S_2=\langle 8,11,18 \rangle$ to the second case, and $S_3=\langle 10,16,27 \rangle $ to the third
one.
\end{ex}

The following example shows that $k$ is not necessarily the maximal
possible value for $\la (T(G(S)))$, when $k>2$ and $G(S)$ is
$k$--Buchsbaum.

\begin{ex}
Let $S=<6,7,16>$. The
Apéry table of $S$ is

\[
\begin{array}{|c|c|c|c|c|c|c|c|c|c|c|c|}\hline
\Ap(S)&0&7&14&21&16&23 \\ \hline \Ap(M)&6&7&14&21&16&23 \\ \hline
\Ap(2M)&12&13&14&21&22&23 \\ \hline
 \Ap(3M)&18&19&20&21&28&29 \\ \hline
\Ap(4M) &24&25&26&27&28&35 \\ \hline \Ap(5M) &30&31&32&33&34&35 \\
\hline
\end{array}
 \]

\medskip

As a consequence, $T=\{16,22,23,29\}$ and
$\ord(x+y)>\ord(x)+\ord(y)$ for all $x\in T$ and $y\in 3M\setminus
4M$. Hence $G(S)$ is $3$--Buchsbaum.  But $\la (T(G(S)))=4$.
\end{ex}

\begin{rem} Note  that when $G(S)$ is
$k$--Buchsbaum and not Cohen-Macaulay, then $T\neq \emptyset$ and
for each element $x\in T$ with maximal expression
$x=r_1n_1+r_2n_2+r_3n_3$, we have $r_2n_2+r_3n_3\in T$ by Remark
\ref{rem3} and $r_1\leq k-1$.  Now set
$$\begin{array}{lll}
x' & = & \max \{ x \mid (k_S-x)n_3 \in T \} \\
y' & = & \max \{ y \mid (k_S-x)n_3 + yn_2 \in T \,\, \textrm{is maximal expression for some} \, x \geq 0 \} \\
\end{array}$$

Observe that by Lemma \ref{kh}, $(k_S-x')n_3+y'n_2 \in T $, and since $G(S)$ is $k$--Buchsbaum it must happen that $x' + y' \leq
k-1$. Let $X = x' + 1, Y = y ' +1$. Then, $X + Y \leq k + 2$ and a
biggest value of the function $f(X,Y) = XY$ under the constraint $X
+ Y \leq k + 2$ gives a bound for the cardinal of the set of
elements that can be written in the above way, which contains $T$.
It is then easy to see that this biggest value is attained for $X =
Y = \frac{k+1}{2}$ and so $$\# \{x\in T \mid x=r_2n_2+r_3n_3 \mbox{
is a maximal expression }\}\leq \frac{(k+1)^2}{4}.$$
 Hence $\la (H^0_{G(S)_+}(G(S)))\leq \frac{k(k+1)^2}{4}$ and $\# (\Ap(S) \cap T ) \leq
\frac{(k+1)^2}{4}$.
\end{rem}

As we have seen in the $2$--Buchsbaum case, the above formula is not a
sharp bound for $\la (H^0_{G(S)_+}(G(S)))$, in the sense that for
$k=2$, we never have the equality in this formula. In fact, the structure of the set $T$ is more involved than the one used to get the formula. Nevertheless, having a good bound for $\la (H^0_{G(S)_+}(G(S)))$ may be useful to detect the $k$-Buchsbaum property of $G(S)$. Hence is natural to ask:
\begin{que}
Assume that $b=3$ and $G(S)$ is $k$--Buchsbaum. Is there a sharp
formula depending on $k$ bounding $\lambda(T(G(S)))$?
\end{que}

\begin{rem}
From the previous bounds we have in the $2$-Buchsbaum case that
$\# (\Ap(S) \cap (T \setminus \{0\})) \leq 2$. This gives an alternative for proving the fact that $(k_S-1)n_3+n_2 \notin \Ap (S)$ in the third case of Theorem \ref{2b}. Namely, if $(k_S-1)n_3+n_2 \in \Ap (S)$, it also belongs to $T$ by Lemma \ref{kh} and so $\# (\Ap(S) \cap T ) \geq 3$, a contradiction.
\end{rem}

\section{Non-decreasing Hilbert functions}




The objective of this  section is to study the growth of the Hilbert
function of a numerical semigroup ring. For that we shall use the
Apéry table structure to provide some new cases where the Hilbert
function is non decreasing. Namely, we shall prove that this
property holds when the embedding dimension is $4$ and $G(S)$ is
Buchsbaum, and when $S$ is balanced, a notion that extends to any
embedding dimension the case considered by Patil-Tamone
in \cite{PT}. On the way, we shall also provide a new and simple
proof for the well known case of embedding dimension $3$.

\medskip

We start by recalling the two subsets in the Apéry table that
control the behavior of the Hilbert function.

\begin{rem}\label{Hi}
Let $k$ be a positive integer. Consider the subsets of $(k-1)M$ defined by
\[ D_k:=\{x \mid \ord(x)=k-1 \textrm{ and } \ord(x+n_1)>k \} \subset
T\]
 and
\[C_k:=\{ y  \mid \ord(y)=k \textrm{ and } y-n_1\notin ( k-1)M \}. \]

\noindent Observe that the elements in $D_k$ have torsion order one.
We have that $\# C_k$ is the number of ends of landings in the row
$k$ of the Apéry table, and that $\# D_k$ is the number of
beginnings of true landings in the row $k$ of the Apéry table. Then,
the successive differences of the Hilbert function of $G(S)$ may be
computed as $\Hi(k)-\Hi(k-1)= \# C_k - \# D_k $.
\end{rem}

In order to prove that $\Hi$ is a non-decreasing function we will
construct an injective map $D_k \longrightarrow C_k$. Next lemma
will allow to construct elements in $C_k$ from elements in $D_k$.

\begin{lem}\label{land}
Assume that $x\in T$ with $\tord(x)=1$ and that
$x+n_1=\sum^b_{i=2}s_in_i$ is a maximal expression. Let
$l_x=\ord(x+n_1)-\ord(x)-1 > 0$. Then
\begin{enumerate}
\item $l_x < \sum^{b-1}_{i=2}s_i$.

\item Let $y_x = \sum^b_{i=2}(s_i-r_i)n_i$, with $0\leq r_i\leq s_i$ for all $i =2, \dots, b$, and such that $\sum^b_{i=2}r_i = l_x$. Then $y_x\in C_k$, where $k=\ord(x)+1$.
\end{enumerate}
\end{lem}
\begin{proof}
(1) Let $x=\sum^b_{i=1}r_in_i$ be a maximal expression with
$\ord(x)=\sum^b_{i=1}r_i$. If $s_b \geq \sum^b_{i=1}r_i+1$, then
\[s_b n_b \geq
(\sum^b_{i=1}r_i+1)n_b>r_1n_1+r_2n_2+\cdots+r_b n_b+n_1=x+n_1,\]
which is a contradiction. Thus $s_b < \sum^b_{i=1}r_i+1$ so that
$\ord(x+n_1)=\sum^b_{i=2}s_i <
\sum^{b-1}_{i=2}s_i+\sum^b_{i=1}r_i+1=\sum^{b-1}_{i=2}s_i+\ord(x)+1$.
Hence $l_x < \sum^{b-1}_{i=2}s_i$.

\medskip

(2) Since $y_x = \sum^b_{i=2}(s_i-r_i)n_i$ is a maximal expression
we have that $\ord(y_x) = \sum_{i=2}^b s_i - \sum_{i=2}^b r_i =
\ord(x+n_1)- l_x =\ord(x) + 1 = k$. Now we have to prove that
$y-n_1\notin (k-1)M$. If $y\in Ap(S)$ then $y-n_1\notin S$. If not,
then $y-n_1 = x-w \in S$, where $w=\sum_{i=2}^b r_in_i$. We have
that $\ord(w) = l_x>0$, so $k = \ord (x) + 1 > \ord (x-w) + 1 =
\ord(y-n_1) + 1$. Hence $k-1 > \ord(y-n_1)$ and $y-n_1 \notin
(k-1)M$.
\end{proof}

The above general construction allows us to show the non-decreasing
of the Hilbert function for numerical semigroup rings in several
cases. For instance, when we apply our strategy to the three
generated case, we obtain a very simple proof which is similar to
the one by I. C. \c{S}erban  in \cite{Se}.

\begin{prop}
Let $b = 3$. Then the numerical semigroup ring $k[[S]]$ has
non-decreasing Hilbert function.
\end{prop}
\begin{proof}

By using the notation of Remark \ref{Hi}, it will
suffice to construct an injective map $D_k \longrightarrow C_k$. Let
$x\in D_k$. By Lemma \ref{rs}, we have that $x=r_1n_1+r_2n_2+r_3n_3$
with $\ord(x)=k-1=r_1+r_2+r_3$ and $x+n_1=s_2n_2+s_3n_3$ with
$\ord(x+n_1)=s_2+s_3>k$ and $s_2\neq 0$. Let
$l_x=\ord(x+n_1)-\ord(x)-1$. Then  $l_x < s_2$ by Lemma
\ref{land}(1) and so we may consider the element
$y_x:=(s_2-l_x)n_2+s_3n_3$. By Lemma \ref{land}(2), $\ord (y_x)=k =
\ord(x)+1$ and  $y_x\in C_k$.

Thus the map
\[ \begin{array}{cll}
D_k &\longrightarrow & C_k\\ x&\mapsto &y_x \end{array}
\]
is well defined.

Now we see that this map is injective. Assume that $x$, $x'$ are
elements in $D_k$ with $y_x=y_{x'}$. Thus, $\ord(x)=\ord(x')=k-1$,
and if we write $x+n_1=s_2n_2+s_3n_3$, $x'+n_1=s'_2n_2+s'_3n_3$ with
$s_2+s_3=\ord(x+n_1)$ and $s'_2+s'_3=\ord(x'+n_1)$, then
$y_x=x+n_1-l_xn_2=x'+n_1-l_{x'}n_2=y_{x'}$. So, assuming for
instance that $l_{x'}\geq l_x$, we will have $x'=x+(l_{x'}-l_x)n_2$.
But since $\ord(x)=\ord(x')$ this is only possible if $l_{x'}-l_x
=0$, and so $x=x'$.
\end{proof}

We may also prove, in general, that if the torsion of $G(S)$ has length one then the Hilbert function is non-decreasing (extending the case Cohen-Macaulay).

\begin{prop}
Let $S$ be a numerical semigroup such that $\la (H^0_{G(S)_+}(G(S)))=1$. Then, $k[[S]]$ has
non-decreasing Hilbert function.
\end{prop}
\begin{proof}
Again by using the notation of Remark \ref{Hi}, it will
suffice to construct an injective map $D_k \longrightarrow C_k$. Note that in this case there is only one torsion element in $G(S)$ and so there is only one possible value $k$ such that $D_k\neq \emptyset$, and only one element $x$ in $D_k$. Hence it suffices to construct an element in $C_k$. Assume that $x + n_1 = \sum _{i=2}^bs_in_i$ is a maximal expression with $l_x= \sum _{i=2}^bs_i - \ord (x) - 1 >0$. Then, $l_x < \sum^b_{i=2}s_i$ and one can find values $r_i \leq s_i$, for any $i=2, \dots , b$, such that $\sum_{i=2}^b r_i = l_x$. Let $y_x=\sum_{i=2}^b(s_i-r_i)n_i$. By Lemma \ref{land}(2), $y_x\in C_k$ and we are done.
\end{proof}

Our next result shows the new case that the Hilbert function of
monomial curves of embedding dimension $4$ whose tangent cone is
Buchsbaum is non-decreasing. First we need to introduce the
following notation.

\begin{defn}\label{r}
Assume that $x\in S$. Define $r_x:=(r_1, \dots , r_b)$, where
$x=\sum^b_{i=1}r_in_i$ is the maximal expression in which
$$\begin{array}{l}r_1=\max\{r'_1 \mid r'n_1 \mbox{is part of a maximal expression of } x \},\\
r_2=\max\{r'_2 \mid r_1n_1 + r'_2n_2 \mbox{ is part of a maximal expression of } x \},\\
... \\
r_b=\max\{r'_b \mid r_1n_1 \! + \! r_2n_2 \! + \! \cdots \! + \! r_{b-1}n_{b-1} \! + \!
r'_bn_b \mbox{ is part of a maximal expression of } x \}.
\end{array}$$

Note that $r_b$ is determined by the previous ones. For $x,y\in S$,
we use $r_x\cdot r_y$ to denote the vector $(r_1 \cdot s_1, \dots ,
r_b \cdot s_b)$, where $r_y=(s_1, \dots, s_b)$. We denote by $0$ the
null vector.

It is easy to see that if $x = \sum _{i=1}^b r_in_i$ with $r_x =
(r_1, \dots , r_b)$ and $x' = \sum _{i=1}^b r'_in_i$ is a
subrepresentation with $r'_i \leq r_i$ for all $i=1, \dots , b$,
then $r_{x'} = (r'_1, \dots , r'_b)$.
\end{defn}

Next lemma allows to construct torsion elements from torsion maximal
expressions.

\begin{lem}
\label{subtorsion} Let $S$ be a numerical semigroup minimally
generated by $n_1< \cdots < n_b$. Let $x=\sum_{i=1}^b r_in_i\in
T$ a maximal expression such that $x + cn_1 = y =
\sum_{i=1}^b s_in_i$, a maximal expression with $\ord (y)> \ord(x) +
c$. Assume that $r_j, s_j \neq 0$ for some $j =1, \dots b$. Then,
$x' = \sum_{i=1, i\neq j}^b r_in_i +(r_j-1)n_j \in T$ with $\tord
(x')\leq c$.
\end{lem}
\begin{proof}
We have that $y' = x' + cn_1= \sum_{i=1, i\neq j}^b r_in_i +(r_j-1)n_j +
cn_1 = \sum_{i=1, i\neq j}^b s_in_i + (s_j-1)n_j$ with $\ord (x') +
c =  \sum_{i=1, i\neq j}^b r_i  +(r_j-1) + c < \sum_{i=1, i\neq j}^b
s_i + (s_j-1) = \ord (y')$ because both $x', y'$ are
subrepresentations of maximal expressions, and so maximal
expressions. Hence $x'\in T$ with $\tord(x')\leq c$.
\end{proof}

\begin{prop}
\label{null} Assume that $G(S)$ is Buchsbaum and let $x$ be a
torsion element. Then $r_x\cdot r_{x+n_1}=0$.
\end{prop}
\begin{proof}
Since $G(S)$ is Buchsbaum, then any torsion element is annihilated
by all elements in $G(S)_+$. Thus $x = \sum^b_{i=1}r_in_i$ cannot have
a torsion subrepresentation. By the above Lemma \ref{subtorsion} we
must have that $\min\{r_i,s_i\}=0$, that is $r_i\cdot s_i=0$ for
$i=1,\ldots,b$, where $(s_1, \dots , s_b) = r_{x+n_1}$.
\end{proof}

For embedding dimension $4$ we have the following:

\begin{prop}
Assume that $b=4$ and that for any $x\in T $,
$\tord(x)=1$. Assume also that $r_x \cdot r_{x +n_1} = 0$ for all
$x\in T$. Then the Hilbert function of $k[[S]]$ is non-decreasing.
\end{prop}
\begin{proof}
Let $x$ be a torsion element with $r_x = (r_1, r_2, r_3, r_4)$ and
$r_{x+n_1} = (s_1, s_2, s_3, s_4)$. First note that since the
torsion order of any torsion element is $1$, then $r_1 = 0$. Also that
$s_1=0$. Now, by Lemma \ref{land}(1)
\begin{equation}\label{e1}
s_2+s_3 > \ord(x+n_1)-\ord(x)-1=l_x. \end{equation}

Let  $l=\min\{l_x,s_2-1\}$, so that $s_2-l>0$ and
$s_3-l_x+l\geq0$. Then we may  define

$$y_x:=\left\{\begin{array}{lll}
x+n_1-l_xn_2 & \mbox{ if } s_3=0 & (a) \\
x+n_1-l_xn_3 & \mbox { if } s_2=0 & (b) \\
 (s_2-l)n_2+(s_3-l_x+l)n_3 & \mbox { if } s_2, s_3\neq 0 & (c)
 \end{array}\right.$$

Observe that if $x\in D_k$, then $y_x\in C_k$, by  Lemma
\ref{land}(2). Thus the map
\[ \begin{array}{cll}
D_k &\longrightarrow & C_k\\ x&\mapsto &y_x \end{array}
\]
is well defined. Let us see that it is injective.

Assume that there exists another torsion element $x'\in D_k$ with
$y_{x'}=y_x$. Let $l_{x'}:=\ord(x'+n_1)-\ord(x')-1$ and
$r_{x'}=(r_1, r_2, r_3, r_4)$, $r_{x'+n_1}=(s'_1, s'_2, s'_3, s'_4)$
be the corresponding vectors (note that $s'_1=0$). By definition of
$D_k$ we have that $\ord(x)=\ord(x') = k-1$.

Now we are going to distinguish several possibilities:

\begin{enumerate}

\item[(i)] If both $y_x$ and $y_{x'}$ are in case (a), and assuming that $l_x\geq l_{x'}$,
then $x+n_1=y_x+l_xn_2=y_{x'}+l_xn_2=x'+n_1+(l_x-l_{x'})n_2$.
Canceling $n_1$ we get that $x = x' + (l_x - l_{x'})n_2$. So
$l_x=l_{x'}$ because $\ord(x)=\ord(x')$, and we get that $x=x'$.

\item[(ii)] If both $y_x$ and $y_{x'}$ are in case (b), then similarly to (i), we have $x=x'$.

\item[(iii)] If both $y_x$ and $y_{x'}$ are in case (c), then by hypothesis we have that $x=r_1n_1+r_4n_4$ and $x'=r'_1n_1+r'_4n_4$. But as noted at the beginning, $r_1 = r'_1 = 0$, so $x=r_4n_4$ and $x'=r'_4n_4$, which implies that $x = x'$ since both are maximal expressions and $\ord(x)=\ord(x')=k-1$.

\item[(iv)] If $y_x$ is in  case  (a) and
$y_{x'}$ is in case (b), then $y_x=(s_2-l_x)n_2+s_4n_4$ and
$y_{x'}=(s'_3-l_{x'})n_3+s'_4n_4$. Hence $s_2-l_x = s'_3-l_{x'} =
0$, by the uniqueness of such maximal expressions. But this is a
contradiction because $s_3 = 0$ and so $s_2 > l_x$.

\item[(v)] If $y_{x'}$ is in case (a) and $y_{x}$ is in case (c), then
$(s'_2-l_{x'})n_2+s'_4n_4=(s_2-l
)n_2+(s_3-l_x+l)n_3$. Once more because of the uniqueness of such
maximal expressions we have that $s_3-l_x+l=0$. Now
$l=s_2-1<l_x$ since $s_3\neq0$.  Hence $y_x=n_2$ which contradicts
the fact that $\ord(y_x)=\ord(x)+1>1$.

\item[(vi)] If $y_{x'}$ is in case (b) and $y_{x}$ is in case (c), then
$(s'_3-l_{x'})n_3+s'_4n_4=(s_2-l)n_2+(s_3-l_x+l)n_3$.
Again because of the uniqueness of the involved maximal expressions
we have that $s_2 - l = 0$, which is a contradiction to our
definition of $l$.
\end{enumerate}
\end{proof}

As a consequence of the above two propositions and taking into
account that if $G(S)$ is Buchsbaum then the torsion order of any
element in $T$ is equal to $1$, we get our desired result:

\begin{thm}
Assume that $b=4$ and $G(S)$ is Buchsbaum. Then, the Hilbert
function of $R$ is non-decreasing.
\end{thm}

\begin{ex}
We quote the following example from \cite{DMS2}. Let $S = <10, 17, 23, 82>$. The Apéry table is

\[
\begin{array}{|c|c|c|c|c|c|c|c|c|c|c|c|}
\hline \Ap(S)&0&17&23&34&46&51&68&69&82&85 \\
\hline \Ap(M)&10&17&23&34&46&51&68&69&82&85 \\
\hline \Ap(2M)&20&27&33&34&46&51&68&69&92&85 \\
\hline \Ap(3M)&30&37&43&44&56&51&68&69&92&85 \\
\hline \Ap(4M)&40&47&53&54&66&61&68&79&92&85 \\
\hline \Ap(5M)&50&57&63&64&76&71&78&89&102&85 \\
\hline \Ap(6M)&60&67&73&74&86&81&88&99&102&95 \\
\hline
\end{array}
 \]

\medskip

We have that $H^0_{G(S)_+}(G(S)) = \{(t^{82})^*, (t^{92})^*\}$, so $\la (H^0_{G(S)_+}(G(S)))=2$, and $G(S)$ is Buchsbaum by \cite[Remark 3.9]{DMS2} (or check it directly by using the Apéry table). The Hilbert function of $k[[S]]$ is $H(n) = \{1,4,5,7,9,9,10\rightarrow \}$.
\end{ex}

Finally, we consider balanced numerical semigroups. The notion of
balanced has been considered in the case of $4$ generated numerical
semigroups by Patil-Tamone in \cite{PT}. Our definition
generalizes their definition to any embedding dimension. We shall
prove that the Hilbert function of any balanced numerical semigroup
is non-decreasing, extending \cite[Theorem 2.11]{PT}.

\begin{defn}
$S$ is called balanced, if $n_i+n_j=n_{i-1}+n_{j+1}$ for all $i\neq
j\in\{2,\ldots,b-1\}$.
\end{defn}

\begin{rem}
$S$ is balanced if and only if $n_i+n_j=n_1+n_{i+j-1}$ for all
$i\neq j\in\{2,\ldots,b-1\}$ with $i+j\leq b-1$.
\end{rem}

In order to make more clear our arguments and for the purposes of this paper, we single out in a new
definition the basic property of balanced numerical semigroups we
shall use.

\begin{defn}
We say $S$ has a cyclic $1$-torsion, if for each element $x\in T$ with
$\tord(x)=1$, there exists $2\leq i\leq b-1$ such that
$x+n_1=s_in_i+s_bn_b$ is the maximal expression that satisfies
Definition \ref{r}.
\end{defn}

It is obvious that all $3$ generated semigroups have a cyclic $1$-torsion. And also
balanced numerical semigroups have a cyclic $1$-torsion.

\begin{prop}
Assume that $S$ is balanced. Then $S$ has a cyclic $1$-torsion.
\end{prop}
\begin{proof}
Let $s\in T$ with $\tord(x)=1$, and let $x + n_1 = \sum _{i=1}^b
s_in_i$ be a maximal expression with $r_{x+n_1} = (s_1, \dots ,
s_b)$. Recall that $s_1 =0$ because $\tord(x)=1$. Assume that there
exist $s_i, s_j \neq 0$ for some $1 < i < j <b$. This will imply in
particular that $y = n_i + n_j$ is also a maximal expression that
satisfies definition \ref{r}. But then, since $n_i + n_j = n_{i-1} +
n_{j+1}$, we have a contradiction because $y = n_{i-1} + n_{j+1}$ is
also a maximal expression.
\end{proof}

Now we have:

\begin{prop}
The Hilbert function of a semigroup $S$ with a cyclic $1$-torsion is non-decreasing.
\end{prop}

\begin{proof}
We will proceed as in the previous cases. Assume that $x\in D_k$ and
let $x+n_1=s_in_i+s_bn_b$ be the corresponding unique maximal
expression that satisfies Definition \ref{r}. Then
$l_x=\ord(x+n_1)-\ord(x)-1 < s_i$ by Lemma 4.2(1). Now, define
$y_x:=(s_i-l_x)n_i+s_bn_b$. By Lemma 4.2(2), $y_x \in C_k$.

Assume that $y_x=y_{x'}$ for some $x,x'\in D_k$, with
$x'+n_1=s'_jn_j+s'_bn_b$ the chosen unique maximal expression. Then
$(s_i-l_x)n_i+s_bn_b=(s'_j-l_{x'})n_j+s'_bn_b$. Hence $s_b = s'_b$
and $i = j$. Now, assuming for instance that $s_i\geq s'_i$, we have
that $x +n_1 = x' + n_1 + (s_i - s'_i)n_i$ and so $x = x' + (s_i -
s'_i)n_i$. But on the other hand $\ord(x)=\ord(x')=k$, hence $x=x'$.
\end{proof}

As a consequence we extend to any embedding dimension the result
proven by Patil-Tamone for embedding dimension $4$, see
\cite[Theorem 2.11]{PT}.

\begin{thm}
Assume that $S$ is  balanced.  Then  the Hilbert function of $R$ is
non--decreasing.
\end{thm}

\begin{ex}
It is not difficult to construct balanced numerical semigroups. The following example is a numerical semigroup having a cyclic $1$-torsion, which is neither balanced or Buchsbaum.  Let $S = <11, 18, 104, 118>$. The Apéry table is

\[
\begin{array}{|c|c|c|c|c|c|c|c|c|c|c|c|c|}
\hline \Ap(S)&0&18&36&54&72&90&104&108&118&122&136 \\
\hline \Ap(M)&11&18&36&54&72&90&104&108&118&122&136 \\
\hline \Ap(2M)&22&29&36&54&72&90&115&108&129&122&136 \\
\hline \Ap(3M)&33&40&47&54&72&90&126&108&140&133&147 \\
\hline \Ap(4M)&44&51&58&65&72&90&126&108&151&144&158 \\
\hline \Ap(5M)&55&62&69&76&83&90&126&108&162&144&169 \\
\hline \Ap(6M)&66&73&80&87&94&101&126&108&162&144&180 \\
\hline \Ap(7M)&77&84&91&98&105&112&126&119&162&144&180 \\
\hline \Ap(8M)&88&95&102&109&116&123&137&130&162&144&180 \\
\hline \Ap(9M)&99&106&113&120&127&134&148&141&162&155&180 \\
\hline \Ap(10M)&110&117&124&131&138&145&159&152&173&166&180 \\
\hline
\end{array}
 \]

\medskip

We have that the $1$-torsion elements of $H^0_{G(S)_+}(G(S))$ are exactly $$\{(t^{115})^*, (t^{133})^*, (t^{151})^*, (t^{169})^*\}$$ and that $126 = 7\cdot 18$, $144 = 8\cdot 18$, $162 = 9\cdot 18$, $180 = 10\cdot 18$ are maximal expressions satisfying Definition \ref{r}. The Hilbert function of $k[[S]]$ is $H(n) = \{1,4,7,7,7,7,7,8,9,10, 11 \rightarrow \}$.
\end{ex}

\medskip
Finally, we would like to thank the referees for a careful reading of the manuscript and several valuable comments and suggestions.


\end{document}